\documentclass[11pt,a4paper,twoside]{amsart}
\usepackage[activeacute,english]{babel}
\usepackage[latin1]{inputenc}
\usepackage{amsfonts}
\usepackage{amsmath,amsthm}
\usepackage{amssymb}
\usepackage{graphics}

\newcommand{\CC}{{\mathcal C}}
\newcommand{\DD}{{\mathcal D}}

\newcommand{\FF}{{\mathcal F}}

\newcommand{\XX}{{\mathcal X}}

\newcommand{\C}{{\mathbb C}}
\newcommand{\N}{{\mathbb N}}
\newcommand{\R}{{\mathbb R}}

\newcommand{\supp}{{\operatorname{supp}}}
\newcommand{\diam}{{\operatorname{diam}}}

\newcommand{\dist}{{\operatorname{dist}}}
\newcommand{\Tr}{{\operatorname{t}}}
\newcommand{\Ran}{{\operatorname{rn}}}
\newcommand{\Ker}{{\operatorname{kr}}}

\newtheorem{teo}{Theorem}[section]

\newtheorem{lema}[teo]{Lemma}
\newtheorem{coro}[teo]{Corollary}
\newtheorem{propo}[teo]{Proposition}

{\theoremstyle{remark} \newtheorem{remark}[teo]{Remark}}

\title[Shell interactions for Dirac operators]{Shell interactions for Dirac operators}
\author[N. Arrizabalaga, A. Mas, L. Vega]{Naiara Arrizabalaga, Albert Mas, Luis Vega}
\date{May, 2013}
\thanks{Arrizabalaga was supported in part by MTM2011-24054. Mas was partially supported by MTM2010-16232, 2009SGR-000420, and IT-305-07.  Vega was partially supported by MTM2011-24054, UFI11/52 and IT-305-07.}
\address{N. Arrizabalaga, A. Mas, and L. Vega.
Departamento de Matem\'aticas, Universidad del Pa\'is Vasco, 48080 Bilbao (Spain)} \email{naiara.arrizabalaga@ehu.es, amasblesa@gmail.com, luis.vega@ehu.es}

\begin{document}

\begin{abstract}
The self-adjointness of $H+V$ is studied, where $H=-i\alpha\cdot\nabla
+m\beta$ is the free Dirac operator in $\R^3$ and $V$ is a measure-valued potential. The potentials $V$ under consideration are given by
singular measures with respect to the Lebesgue measure, with special
attention to surface measures of bounded regular domains. The existence
of non-trivial eigenfunctions with zero eigenvalue naturally appears in our approach,
which is based on well known estimates for the trace operator defined
on classical Sobolev spaces and some algebraic identities of the Cauchy operator associated to $H$.\\

\noindent
R\'ESUM\'E. Nous \'etudions le caract\`ere auto-adjoint de $H+V$, o\`u  $H=-i\alpha\cdot\nabla
+m\beta$ est l'op\'erateur de Dirac libre dans $\R^3$ et $V$ est un potentiel \`a valeur mesure. Les 
potentiels $V$ consider\'es sont donn\'es par mesures singuli\`eres par rapport \`a la mesure de Lebesgue, 
avec attention particuli\`ere pour le cas des mesures de surface de domaines born\'es r\'eguliers. L'existence de fonctions propres non triviales \`a valeur propre nulle appara\^{i}t de fa\c con naturelle dans notre approche, qui est bas\'ee sur des estimations connues pour l'op\'erateur trace d\'efini dans les espaces de Sobolev classiques et quelques identit\'es alg\'ebriques de l'op\'erateur de Cauchy associ\'e \`a $H$. 

\end{abstract}
\maketitle

\section{Introduction}
In this article we investigate the self-adjointness in $L^2(\R^3)^4$
of the free Dirac operator
\begin{equation*}
H=-i\alpha\cdot\nabla+m\beta\quad(\text{for }m>0)
\end{equation*}
coupled with measure-valued potentials, where $\alpha= (\alpha_1,\alpha_2,\alpha_3)$ and $\alpha_j$ for $j=1,2,3$ and $\beta$ denote
the so-called Dirac matrices (see (\ref{s3eq1}) in Section \ref{s2}
for the details about $H$). Recall that $H$ acts on spinors $\varphi(x)=
\binom{\phi} {\chi}(x)$ with $x\in \R^3$ and $\phi(x),\chi(x)\in
\C^2$. Moreover, 
$H$ is invariant under translations and, for $m=0$, it is also
invariant under scaling because, if
$$\varphi_\lambda(x)=\lambda^{-1}\varphi(\lambda x)\quad\text{for } \lambda>0,
$$
then $H \varphi_\lambda(x)= H \varphi (\lambda x).$
We are interested on critical perturbations of $H$, i.e., those given by potentials $V(x)$ such that, when measured in an appropriate function space, the
rescaled potentials
\begin{equation}\label{Ieq2}
V_{\lambda}(x)=\lambda V(\lambda x)\quad\text{for }\lambda>0
\end{equation}
also belong to the same space and have the same size. We shall pay
special attention to potentials given by measures $\sigma$ such that 
\begin{equation}\label{Ieq3}
\sigma (B)\leq C\diam(B)^2
\end{equation}
for any ball $B\subset\R^3$ (in particular, $\sigma$ and the Lebesgue measure in $\R^3$ are mutually singular),
and more precisely to surface measures of bounded regular domains. Note that, for balls centered at the origin, (\ref{Ieq3}) is invariant under the scaling given by (\ref{Ieq2}) in the distributional sense.

The main question that we want to address is the following: {\em in which sense these critical
perturbations can be considered small with respect to the free Dirac
operator $H$?}. This can be seen as a very first step to understand
more complicated settings, like for example those where $V$
is given by a non-linear potential.
At this respect it is worth mentioning that, as far as we know, all the
available results concerning non-linear Dirac equations involve, in
one way or another, some kind of smallness either on the size of the
initial data or on the time of existence (see \cite{EV}, \cite{MNNO},
\cite{Cacc}).

The first kind of perturbation one could think about is the one given by
potentials $V$ that are hermitian and that grow like the
classical Coulomb potential, that is
$$|x||V(x)|\leq \nu\quad\text{for some }\nu\geq0.$$
For $\nu<1$, there exists a unique domain $D$ where $H+V$
is  selfadjoint and  such that $D$ is a subspace of the space of spinors with finite kinetic
and potential energy, i.e., 
$$D\subset \left\{\varphi\in L^2(\R^3)^4\,:\,\big(I_4-\Delta\big)^{1/4}(\varphi)\in
L^2(\R^3)^4\,\text{ and }\, \int |\varphi|^2\,\frac{dx}{|x|}<+\infty\right\},$$
where $I_4$ denotes the identity operator on $L^2(\R^3)^4$
(see \cite{KW}). It is well known
that, for $V(x)=\nu/|x|$ and $|\nu|>1$, the hamiltonian is not essentially self-adjoint (see \cite{Thaller}),
and it does not seem to exist a natural choice among all the possible
extensions. Concerning the critical case $\nu=\pm 1$, little
is known. For scalar potentials
$$V(x)=v(x)I_4\quad\text{with }v(x)\in\R,$$
partial results have been obtained in \cite{EL}.
The existence of a threshold at $\nu=1$ is a consequence of a sharp
inequality of Hardy type that involves $H$ instead of the usual
gradient. 
Note that $H$ does not leave invariant the set of radial spinors, hence this Hardy's inequality is not a straightforward extension of the
classical one.
Besides, recall that $H$ is not a semibounded operator. In fact, assume
that $V(x)=V(-x)$ and that $\varphi(x)=
\binom{\phi} {\chi}(x)$
is an eigenfunction with eigenvalue $\lambda$. Then
$\widetilde\varphi(x)=
\binom{\chi}{\phi} (-x)$
is an eigenfunction with eigenvalue $-\lambda$. This elemental
property plays a role in one of the main results in
this paper, namely Theorem \ref{t2}.

Motivated by the examples of potentials with Coulombic type
singularities, we want to investigate the case of potentials with
a singular support on a hypersurface $\Sigma\subset\R^3$; spheres and hyperplanes are fundamental
examples. One may
assume without loss of generality that the sphere is
$$S^2=\left\{x\in\R^3\,:\,|x|=1\right\}$$
and the hyperplane is
$$\R^2\times\{0\}=\left\{(x_1, x_2, x_3)\in\R^3\,:\,x_3=0\right\}.$$

The case of the sphere has been studied by different authors like, for example, \cite{Naiara}, \cite{Dittrich}, \cite{Dominguez}, and \cite{Shabani} (see also the references there in). The closest results to
ours are those obtained in \cite{Dittrich}, where a wide variety of hamiltonians
given by potentials $V$ supported on $S^2$ are considered. Besides, spectral
questions of these hamiltonians are studied.
One of the interesting features proved in \cite{Dittrich} is that it seems to
be no size condition at all on the potential $V$ which prevents
from self-adjointness. 

The
approach in \cite{Dittrich} heavily relies on the use of the spherical symmetry and, in particular, on the decomposition in spherical harmonics.
From our point of view, the use of spherical harmonics has the strong limitation that the domain of definition of the
hamiltonians is not explicit. This drawback does not exist in our
approach, which it is essentially based on the use of the trace
inequality for functions of the classical Sobolev space $W^{1,2}(\R^3)$ which will be introduced later on. As a consequence, for proving
self-adjointness, we do not make any particular use of any symmetry, and
our result holds for quite general $\Sigma.$

Regarding the case of the hyperplane $\{x_3=0\}$, Fourier analysis is available and provides a simpler approach. In fact, the domain
of definition of the hamiltonians under study is completely
explicit, as it will be seen. Moreover, it becomes evident the existence of some
critical values for some specific potentials. These critical values play a
fundamental role and, as far as we know, they have been completely
overlooked in previous works.

For the purpose of this introduction, let us focus on the case where
the potential is a $\delta$-shell on $\Sigma=\R^2\times\{0\}\subset\R^3$ (see Proposition \ref{p2} below for
the details). In order to define the hamiltonian
$$H+\lambda\delta_\Sigma\quad \text{with } \lambda\in\R, $$
it is natural to consider a subspace of
$$W^{1,2}(\R^3)^4+\left\{\phi * g\sigma\, :\,g\in L^2(\Sigma)^4\right\}\subset L^2(\R^3)^4$$
as the domain of definition. In that statement, $\phi$ denotes the fundamental solution of $H$ (see Section \ref{s2}) and $g\sigma$ is a singular measure with support on $\Sigma$ and given by an $L^2(\Sigma)^4$ density
$$\int_{\R^2}\big|g(x_1,x_2)\big|^2 \,dx_1 dx_2<+\infty.$$
It is not hard to prove that $\phi* g\sigma$ has a jump discontinuity
on the hyperplane $\Sigma$. Hence, following \cite{Dittrich}, we may define
$$\lambda\delta_{\Sigma}
(\phi\ast g\sigma)=\frac{\lambda}{2}\Big(C_+(g)+C_-(g)\Big),$$
where
$$C_{\pm}(g)(x_1,x_2)=\lim_{x_3\to 0^{\pm}}(\phi* g\sigma)(x_1,x_2,x_3).$$
It turns out that, for this type of hamiltonians, there is a big
difference whether $\lambda=\pm 2$ or not. 
Actually, the domain of definition of the hamiltonian  for $\lambda=
\pm 2$ is completely different from the case of $\lambda\neq\pm2$ (see Proposition \ref{p2} for the details), which is a consequence of (\ref{p2eq5}) below. Moreover,
if $\lambda=\pm2$ then there exist non-trivial eigenfunctions with
eigenvalue zero, something that does not hold if $\lambda\neq\pm2$. Let us mention that the values $\lambda=\pm 2$ appear independently of the hypersurface $\Sigma$, they actually come from the well known identity
\begin{equation}\label{Ieq1}
-4\big(C_{\sigma}(\alpha\cdot N)\big)^2=I_4
\end{equation}
proved in Lemma \ref{l5}, where $N$ denotes  the (exterior) unit
normal vector field to $\Sigma$ and $C_{\sigma}$ is the Cauchy operator associated to $H$ and $\Sigma$.

Nevertheless, the use of Fourier analysis in the case
$\Sigma=\{x_3=0\}$ hides the fundamental algebraic property that allows us to obtain such explicit results. In fact, the hyperplane has the
property that the anticommutator
\begin{equation*}
\left\{\alpha\cdot N,C_{\sigma}\right\}=(\alpha\cdot N)C_{\sigma}
+C_{\sigma}(\alpha\cdot N)
\end{equation*}
is identically zero. It is easy to see that $\left\{\alpha\cdot N,C_{\sigma}\right\}$
does not vanish in general (the case of the sphere is an example, see Proposition \ref{p2}). This feature of the anticommutator is equivalent
to a spectral property of the hamiltonian
$H+\lambda\delta_\Sigma$, namely, the existence of non-trivial
eigenfunctions with eigenvalue zero. We think that this is a relevant
connection and it can be seen as an extra bonus of our approach with
respect to those available in the literarure.
The situation is particularly simpler when $\Sigma$ is the
boundary of a regular bounded domain, because in this case $\left\{
\alpha\cdot N,C_{\sigma}\right\}$ is a compact operator and Fredholm
theory applies. In fact, the eigenvalues of this compact
operator can be written in terms of those $\lambda$'s for which a non-trivial eigenfunction either for $H+\lambda \delta_{\Sigma}$ or for $H-
\lambda \delta_{-\Sigma}$ exists, where 
$-{\Sigma}=\{x\in\R^3:\, -x\in\Sigma\}$. We should mention that this happens as long as $\lambda\neq \pm2$. In this article, the case $\lambda= \pm2$ is only considered for $\Sigma=\R^2\times\{0\}$, and it is unclear what can be expected for general surfaces (including the sphere). We think that this is a relevant open problem and we plan to work on it in the future.

Regarding the results of this article,  the ambient Hilbert space is
$L^2(\R^3)^4$ with respect to the Lebesgue measure, and $H$ is defined
in the sense of distributions. For suitable singular measures $\sigma$
and $L^2(\sigma)^4$-valued potentials $V$ we find domains $D\subset
L^2(\R^3)^4$ in which $H+V$ is an unbounded self-adjoint operator. As
in the case of the hyperplane, our approach is based on the fact that,
if $$\varphi\in D\subset L^2(\R^3)^4\quad\text{and}\quad(H+V)(\varphi)
\in L^2(\R^3)^4,$$ then $H(\varphi)$ has to be the sum in the sense of
distributions of a function $G\in L^2(\R^3)^4$ and another in $g\in
L^2(\sigma)^4$, because of $V$. By the same reason, $V(\varphi)$
should coincide with $-g$. Therefore, $\varphi$ should be the
convolution $\phi*(G+g)$. To guarantee that $H+V$ is symmetric on $D
$, we impose some relations between $G$ and $g$, but these relations
must not be too strong because, for obtaining self-adjointness, $D$
can not be too small.

In this direction, our first main result is Theorem \ref{t1}, which
deals not only with $H$ but with general symmetric differential
operators $L$ on $L^2(\R^n)^b$ with constant coeffitients ($n\geq2$, $b
\geq1$). With the aid of bounded self-adjoint operators $
\Lambda:L^2(\sigma)^b\to L^2(\sigma)^b$, in Theorem \ref{t1} we
construct domains $D$ where $L+V$ is self-adjoint (or essentially self-adjoint), basically by relating $G$ and $g$ through $\Lambda$ for all $
\varphi=\phi*(G+g)\in D$. Depending on the relations that we impose,
we have to require extra properties on $\Lambda$ to ensure self-adjointness.
In this theorem, $V$ is taken so that $L+V:D\to L^2(\R^3)^b$. Indeed,
Theorem \ref{t1} can be considered as a method for constructing self-adjoint extensions of the differential operator $L$ initially defined
on $\CC^\infty_c(\supp(\sigma)^c)$, due to the fact that $V$ lives in $
\supp(\sigma)$ and thus it vanishes on the latter function space.

Our second main result in this article is Theorem \ref{t2}, where the
case of the Dirac operator $H$ coupled with specific potentials $V$
living on the boundary of a regular bounded domain $\Omega$ is
treated. In this setting, the functions $\varphi=\phi*(G+g)$ have
boundary values $\varphi_\pm$ when we approach to $\partial\Omega=
\Sigma$ from inside/outside $\Omega$. The potentials under
consideration in Theorem \ref{t2} are
$$V_\lambda(\varphi)= \lambda\delta_\Sigma(\varphi)=\frac{\lambda}
{2}\,(\varphi_++\varphi_-)$$ for $\lambda\in\R$, that is to say the 
$\delta$-shell potentials that we have mentioned above in the case of the sphere and the hyperplane.
The existence of the critical values $\lambda=\pm2$ and some $\lambda_j$'s for which $\Ker(H+V)\neq{0}$ are described in the statement of the theorem. As we already said, this latter
property is a consequence of the fact that $\left\{\alpha\cdot
N,C_{\sigma}\right\}$ is not trivial and compact. 

In Theorem \ref{t4} we consider potentials defined in terms of some commutativity property. In particular, the theorem applies to some particular magnetic potentials (see $(i)$ and $(ii)$ in Subsection \ref{ss6}). Theorem \ref{t3} is devoted to general potentials satisfying a smallness condition (see $(iii)$, $(iv)$, and $(v)$ in Subsection \ref{ss6} for some examples).

Concerning the structure of the article, Section \ref{s1} is devoted
to construct the aforementioned self-adjoint extensions of symmetric
differential operators, which are interpreted as a coupling with
suitable measure-valued potentials, by using a fundamental solution
and singular measures. Subsections \ref{ss4}, \ref{ss1}, and \ref{ss3}
contain the preliminaries, and the main result of the section is in
Subsection \ref{ss2}, namely Theorem \ref{t1}.

In Section \ref{s2} we focus our attention to the Dirac operator. The
first part of the section contains some basics on its fundamental
solution as well as a direct application of Theorem \ref{t1} to the
Dirac operator coupled with quite general measure-valued potentials.
Subsection \ref{ss5} is devoted to the study of potentials living on
the boundary of a Lipschitz domain. We first provide
some preliminaries about boundary values (such as Plemelj-Sokhotski
jump formulae) and a proof of (\ref{Ieq1}), as well as some other properties of $C_{\sigma}$. Afterwards, we show
the three main results of the subsection, namely Theorems \ref{t2},
\ref{t4}, and \ref{t3} above-mentioned. Proposition \ref{p2} contains
some particularities of Theorem \ref{t2} for the case of the plane and
the sphere. Finally, Subsection \ref{ss6} provides some examples of
potentials which fit in Theorems \ref{t4} and \ref{t3}.

\section{Self-adjoint extensions of symmetric differential operators}\label{s1}
\subsection{Basic definitions and notation}\label{ss4}
For the sequel, $C>0$ denotes a constant which may change its value at different occurrences. 
Throughout this section $n,b>0$ are fixed integers and $d$ is real number such that $0<d<n$, unless we specify something different at each particular situation. 

Given a positive Borel measure $\nu$ in $\R^n$, set
$$L^2(\nu)^b=\left\{f:\R^n\to\C^b\text{ $\nu$-measurable}:\, 
\|f\|^2_{L^2(\nu)^b}:=\int|f|^2\,d\nu<\infty\right\},$$
and denote by $\langle\cdot,\cdot\rangle_{\nu}$ and $\|\cdot\|_\nu$ the standard scalar product and norm in $L^2(\nu)^b$, i.e., 
$\langle f,g\rangle_\nu=\int f\cdot\overline g\,d\nu$ and 
$\|f\|^2_\nu=\int|f|^2\,d\nu$ for $f,g\in L^2(\nu)^b$. Set
$\DD=\CC^\infty_c(\R^n)^b$ ($\C^b$-valued functions defined in $\R^n$ and which are $\CC^\infty$ with compact support), and
$\DD^*$ denotes the space of distributions in $\R^n$ with respect to space of test functions $\DD$. We write $I_b$ or $1$ interchangeably to denote the identity operator on $L^2(\nu)^b$.

Let $\mu$ denote the Lebesgue measure in $\R^n$. Given a Borel measure $\sigma$ in $\R^n$, we say that $\sigma$ is $d$-dimensional if $\sigma(B(x,r))\leq Cr^d$ for all $x\in\R^n$, $r>0$. We also say that $\sigma$ is $d$-dimensional Ahlfors-David regular, or simply $d$-AD regular, if $C^{-1}r^d\leq\sigma(B(x,r))\leq Cr^d$ for all $x\in\supp(\sigma)$ and $0<r<\diam(\supp(\sigma))$.

Finally, we introduce the auxiliary space $$\XX=\left\{G\mu+g\sigma:\,G\in L^2(\mu)^b,\, g\in L^2(\sigma)^b\right\}\subset\DD^*.$$

\subsection{Symmetric differential operators and fundamental solutions}\label{ss1}
Let $L:\DD\to\DD$ be a differential operator with constant coefficients and symmetric with respect to $\langle\cdot,\cdot\rangle_{\mu}$, that is
$\langle L(f),g\rangle_{\mu}
=\langle f,L(g)\rangle_{\mu}$ for all $f,g\in\DD.$
By duality, $L$ is also defined in $\DD^*$, thus we also have that $L:\DD^*\to\DD^*$.

Let $\phi=(\phi_{j,k})_{1\leq j,k\leq b}$ be a (matrix-valued) fundamental solution of $L$, so $\phi*L(f)=f$ for all $f\in\DD$. As usual, we denote by $\phi^t$ the transpose of $\phi$ and by $\overline{\phi}$ the complex conjuate of $\phi$, that is, $(\phi^t)_{\,j,k}=\phi_{k,j}$ and $(\overline{\phi})_{j,k}=\overline{\phi_{j,k}}$ for all $1\leq j,k\leq b$. For the sequel, we assume that $\phi$ satisfies the following conditions:
\begin{itemize}
\item[$(i)$] $\phi_{j,k}\in\CC^\infty(\R^n\setminus\{0\})$ for all $1\leq j,k\leq b$,
\item[$(ii)$] $\phi(x-y)=\overline{\phi^t}(y-x)$ for all $x,y\in\R^n$ such that $x\neq y$,
\item[$(iii)$] there exist $\gamma,\delta>0$ and $0<s<n$ such that
\begin{itemize}
\item[$(a)$] $\sup_{1\leq j,k\leq b}|\phi_{j,k}(x)|\leq C|x|^{-n+s}$ for all $|x|<\delta$,
\item[$(b)$] $\sup_{1\leq j,k\leq b}|\phi_{j,k}(x)|\leq Ce^{-\gamma|x|}$ for all $|x|>1/\delta$,
\item[$(c)$] $\sup_{1\leq j,k\leq b}\,
\sup_{\xi\in\R^n}(1+|\xi|^2)^{s/2}|\FF(\phi_{j,k})(\xi)|<\infty.$
\end{itemize}
\end{itemize}
We should mention that $s$ corresponds to the order of $L$. It is an exercise to show that $(b)+(c)$ imply $(a)$, but we state $(a)$ separately because we are going to use it explicitely in the sequel.
Furthermore, $(ii)$ is a consequence of the fact that $L$ is symmetric, and one formally has
$$\langle f,g\rangle_{\mu}
=\langle \phi*L(f),g\rangle_{\mu}
=\langle L(f),\phi*g\rangle_{\mu}
=\langle f,L(\phi*g)\rangle_{\mu}$$
for all $f,g\in\DD,$
thus $\phi*L(f)=L(\phi*f)=f$ for all $f\in\DD$.

Given a positive Borel measure $\nu$ in $\R^n$, $f\in L^2(\nu)^b$, and $x\in\R^n$, we set $$(\phi*f\nu)(x)=\int\phi(x-y)f(y)\,d\nu(y),$$ whenever the integral makes sense. Observe that $(\phi*f\nu)(x)$ is a vector of $b$ components.

\subsection{Preliminary results}\label{ss3}
This section is devoted to prove auxiliary lemmata necessary to state and prove the main result of Section \ref{ss2}.  

\begin{lema}\label{l1}
Let $\nu$ be a $d$-dimensional measure in $\R^n$ with $0<d\leq n$. 
If $2s>n-d$, then $\|\phi*g\nu\|_{\mu}\leq C\|g\|_{\nu}$ for all $g\in L^2(\nu)^b$.
\end{lema}
\begin{proof}
Set $K(x)=\sup_{1\leq j,k\leq b}|\phi_{j,k}(x)|$ for $x\in\R^n\setminus\{0\}$. Let $\epsilon$ be such that $$\max\{0,d-2n+2s\}<\epsilon<\min\{d,d-n+2s\}.$$ Then, by Cauchy-Schwarz inequality,  
\begin{equation}\label{l1eq1}
\begin{split}
|(\phi*g\nu)(x)|^2&\leq C\bigg(\int K(x-y)|g(y)|\,d\nu(y)\bigg)^2\\
&\leq C\bigg(\int K(x-y)^{\frac{d-\epsilon}{n-s}}
\,d\nu(y)\bigg)\bigg(\int K(x-z)^{2-\frac{d-\epsilon}{n-s}}|g(z)|^2\,d\nu(z)\bigg).
\end{split}
\end{equation}
By $(iii)(a)$ and $(iii)(b)$ in Section \ref{ss1}, $K(x-y)^{(d-\epsilon)/(n-s)}\leq C|x-y|^{-d+\epsilon}$ for $|x|<\delta$ and 
$K(x-y)^{(d-\epsilon)/(n-s)}\leq Ce^{-\gamma(d-\epsilon)(n-s)^{-1}|x-y|}$ for $|x|>1/\delta$ (notice that $\gamma(d-\epsilon)/(n-s)>0$ since $n>s$ by $(iii)$ in Section \ref{ss1}). Hence, using that $\sigma$ is a $d$-dimensional measure in $\R^n$ and integration in dyadic annuli, we easily deduce that 
\begin{equation}\label{l1eq2}
\sup_{x\in\R^n}\int K(x-y)^{\frac{d-\epsilon}{n-s}}\,d\nu(y)<\infty.
\end{equation} 
Similarly, by $(iii)(a)$ and $(iii)(b)$ in Section \ref{ss1},
$K(x-y)^{2-(d-\epsilon)/(n-s)}\leq C|x-y|^{-2n+2s+d-\epsilon}$ for $|x|<\delta$ (recall that $2n-2s-d+\epsilon<n$) and $K(x-y)^{2-(d-\epsilon)/(n-s)}\leq Ce^{-\gamma(2-(d-\epsilon)/(n-s))|x-y|}$ for $|x|>1/\delta$ (notice that $\gamma(2-(d-\epsilon)/(n-s))>0$). Since $\mu$ is $n$-dimensional, we have
\begin{equation}\label{l1eq3} 
\sup_{z\in\R^n}\int K(x-z)^{2-\frac{d-\epsilon}{n-s}}\,d\mu(x)<\infty.
\end{equation} 
Therefore, combining (\ref{l1eq1}), (\ref{l1eq2}), Fubini's theorem, and (\ref{l1eq3}), we conclude that
\begin{equation*}
\begin{split}
\|\phi*g\nu\|^2_{\mu}&=\int|(\phi*g\nu)(x)|^2\,d\mu(x)\leq C\iint K(x-z)^{2-\frac{d-\epsilon}{n-s}}|g(z)|^2\,d\nu(z)\,d\mu(x)\\
&\leq C\int|g(z)|^2\,d\nu(z)=C\|g\|^2_{\nu}.
\end{split}
\end{equation*}
and the lemma is proved.
\end{proof}

\begin{remark}
The assumption $(iii)(b)$ in Section \ref{ss1} can be easily relaxed for the purposes of Lemma \ref{l1}, but we will not go further in this direction.
\end{remark}

\begin{coro}\label{c1}
Let $\sigma$ be a $d$-dimensional measure in $\R^n$ with $0<d\leq n$, and assume $2s>n-d$. For $G\mu+g\sigma\in\XX$, set $$\Phi(G+g)=
\phi*G\mu+\phi*g\sigma.$$
Then 
$\|\Phi(G+g)\|_\mu\leq C(\|G\|_\mu+\|g\|_\sigma)$ for all $G\mu+g\sigma\in\XX$.
\end{coro}
\begin{proof}
Apply Lemma \ref{l1} to $\nu=\sigma$ and to $\nu=\mu$ separately.
\end{proof}

\begin{lema}\label{l3}
Let $\sigma$ be a $d$-dimensional measure in $\R^n$ with $0<d\leq n$,
and assume $2s>n-d$.
For every $G\mu+g\sigma\in\XX$, $L(\Phi(G+g))=G\mu+g\sigma$ in the sense of distributions.
\end{lema}
\begin{proof}
Recall that $L(\phi*f)=\phi*(Lf)$ for all $f\in\DD$. The lemma follows easily by the fact that $L$ is symmetric and that $\phi$ satisfies $(ii)$ in Section \ref{ss1}.
\end{proof}

\begin{coro}\label{c3}
Let $\sigma$ be a $d$-dimensional measure in $\R^n$ with $0<d<n$,
and assume $2s>n-d$.
Given $G\mu+g\sigma\in\XX$ and $\varphi=\Phi(G+g)$, set  
\begin{equation*}
V(\varphi)= -g\sigma\quad\text{and}\quad
L_V(\varphi)=L(\varphi)+V(\varphi).
\end{equation*}
Then $V$ is well defined. Moreover,  $L_V(\varphi)=G\mu$ in the sense of distributions. For simplicity of notation, we write $L_V(\varphi)=G\in L^2(\mu)^b$.
\end{coro}
\begin{proof}
Assume that $\varphi=\Phi(G+g)=\Phi(F+f)$ for some $G\mu+g\sigma,F\mu+f\sigma\in\XX$. By Lemma \ref{l3},  
$G\mu+g\sigma=L(\Phi(G+g))=L(\Phi(F+f))
=F\mu+f\sigma$ in the sense of distributions.
Since $d<n$, $\mu$ and $\sigma$ are mutually singular, and
we easily deduce that $G=F$ in $L^2(\mu)^b$ and $g=f$ in $L^2(\sigma)^b$. Hence
$V(\varphi)=-g\sigma=-f\sigma$, so $V$ is well defined. Furthermore, 
$L_V(\varphi)=G\mu+g\sigma-g\sigma=G\mu$ distributionally, which finishes the proof.
\end{proof}

The next proposition states some known results on the trace of functions of Sobolev spaces. Let $W^{r,2}(\mu)^b$ be the Sobolev space of $\C^b$-valued functions such that all its components have all its derivatives up to order $r>0$ in $L^2(\mu)$.

\begin{propo}\label{p1}
Let $\Sigma\subset\R^n$ be closed, let $0<d<n$ and $\sigma$ be the $d$-dimensional Hausdorff measure restricted to $\Sigma$, and assume that $\sigma$ is $d$-AD regular.  
For $G\in\DD$ consider the trace operator $\Tr_\Sigma(G)=G\chi_{\Sigma}$. If $r>0$ is such that $2r>n-d$, then $\Tr_\Sigma$ extends to a bounded linear operator $\Tr_\sigma:W^{r,2}(\mu)^b\to L^2(\sigma)^b$ in the following cases:
\begin{itemize}
\item[$(i)$] if $d>n-1$,
\item[$(ii)$] if $\Sigma$ preserves Markov's inequality (see \cite{Wallin} for the precise definition),
\item[$(iii)$] if $d\in\N$ and $\Sigma$ is a $d$-dimensional compact $\CC^\infty$ manifold in $\R^n$,
\item[$(iv)$] if $\Sigma$ is either the boundary of a bounded Lipschitz domain in $\R^n$ (i.e. $d=n-1$) or the graph of a Lipschitz function from $\R^{n-1}$ to $\R$.
\end{itemize}
\end{propo}
\begin{proof}
The cases $(i)$ and $(ii)$ are a direct consequence of \cite[Propositions 2 and 4]{Wallin}. The case $(iv)$ follows by \cite{Marschall}, and $(iii)$ can be obtained by the arguments in \cite[Corollary 6.26]{Folland}.
\end{proof}

\begin{remark}
It is known that the trace operator $\Tr_\Sigma$ extends to a bounded linear operator from $W^{r,2}(\mu)^b$ to $L^2(\sigma)^b$ in other cases besides the ones in Proposition \ref{p1}. However, the already mentioned ones are enough for our purposes (in particular $(iv)$). Let us also mention that the trace operator fails to be bounded for $2r=n-d$ even for $d$-planes in $\R^n$, so the condition $2r>n-d$ is sharp in this sense.
\end{remark}

\begin{lema}\label{l2}
We have $\|\Phi(G)\|_{W^{s,2}(\mu)^b}\leq C\|G\|_{\mu}$ for all $G\in L^2(\mu)^b$. 
\end{lema}
\begin{proof}
It follows by $(iii)(c)$ in Section \ref{ss1} and a direct application of Plancherel's theorem. 
\end{proof}

\begin{coro}\label{c2}
Let $\Sigma$ and $\sigma$ be as in any of the cases in Proposition \ref{p1} with
 $2s>n-d$.
For $G\in L^2(\mu)^b$, set $$\Phi_\sigma(G)=\Tr_\sigma(\Phi(G))=\Tr_\sigma(\phi*G\mu).$$ Then, $\Phi_\sigma$ is well defined and
$\|\Phi_\sigma(G)\|_\sigma\leq C\|G\|_\mu$ for all $G\in L^2(\mu)^b$. 
\end{coro}

\begin{proof}
Use Lemma \ref{l2} and Proposition \ref{p1} with $r=s$.
\end{proof}

\begin{lema}\label{l4}
Let $\Sigma$ and $\sigma$ be as in any of the cases in Proposition \ref{p1} with
$2s>n-d$. Then, for every $F\mu,G\mu,g\sigma\in\XX$, we have
\begin{equation*}
\langle \Phi(G),F\rangle_\mu
=\langle G,\Phi(F)\rangle_\mu\quad\text{and}\quad
\langle \Phi(g),F\rangle_\mu
=\langle g,\Phi_\sigma(F)\rangle_\sigma.
\end{equation*}
\end{lema}
\begin{proof}
By $(ii)$ in Section \ref{ss1}, Lemma \ref{l1}, and Fubini's theorem we have
$\int(\phi*G\mu)\overline F\,d\mu=\int G\overline {(\phi*F\mu)}\,d\mu$
for all $F,G\in L^2(\mu)^b$, which means  
$\langle \Phi(G),F\rangle_\mu
=\langle G,\Phi(F)\rangle_\mu$. 

Let us now prove that $\langle \Phi(g),F\rangle_\mu
=\langle g,\Phi_\sigma(F)\rangle_\sigma$. Given $\epsilon>0$, set $$\Omega_\epsilon=\{x\in\R^n:\, |x|<1/\epsilon,\,\dist(x,\Sigma)>\epsilon\}$$ and
$F_\epsilon= F\chi_{\Omega_\epsilon}\in L^1(\mu)^b\cap L^2(\mu)^b.$
Then $\Phi_\sigma(F_\epsilon\mu)(y)=\int_{\Omega_\epsilon}\phi(y-x)F(x)\,d\mu(x)$ for all $y\in\Sigma$ (the integral converges absolutelly).
By $(ii)$ in Section \ref{ss1} and Fubini's theorem,
\begin{equation}\label{l4eq1}
\begin{split}
\langle \Phi(g),F_\epsilon\rangle_\mu
&=\int_{\Omega_\epsilon}\int\phi(x-y)g(y)\cdot\overline{F(x)}\,d\sigma(y)\,d\mu(x)\\
&=\int\int_{\Omega_\epsilon}
g(y)\cdot\overline{\phi(y-x)F(x)}\,d\mu(x)\,d\sigma(y)
=\langle g,\Phi_\sigma(F_\epsilon)\rangle_\sigma.
\end{split}
\end{equation}
Corollary \ref{c1} yields
$|\langle \Phi(g),F-F_\epsilon\rangle_\mu|\leq C
\|g\|_\sigma\|F-F_\epsilon\|_\mu$,
and by Corollary \ref{c2} we have
$|\langle g,\Phi_\sigma(F_\epsilon-F)\rangle_\sigma|
\leq C\|g\|_\sigma\|F-F_\epsilon\|_\mu$.
Therefore, using the triangle inequality and (\ref{l4eq1}), 
\begin{equation*}
\begin{split}
|\langle \Phi(g),F\rangle_\mu 
-\langle g,\Phi_\sigma(F)\rangle_\sigma|
&\leq|\langle \Phi(g),F-F_\epsilon\rangle_\mu|
+|\langle g,\Phi_\sigma(F_\epsilon-F)\rangle_\sigma|
\leq C\|g\|_\sigma\|F-F_\epsilon\|_\mu.
\end{split}
\end{equation*}
The lemma follows by taking $\epsilon\searrow0$ and dominate convergence.
\end{proof}

\subsection{Main result}\label{ss2}

For the rest of this section, we assume that $\Sigma$ and $\sigma$ are as in any of the cases in Proposition \ref{p1} with $2s>n-d$. Given an operator between vector spaces $S:X\to Y$, denote $\Ker(S)=\{x\in X:\, S(x)=0\}$ and $\Ran(S)=\{S(x)\in Y:\, x\in X\}$.

\begin{teo}\label{t1}  
Let $\Lambda:L^2(\sigma)^b\to L^2(\sigma)^b$ be a bounded linear self-adjoint operator.  
\begin{itemize}
\item[$(i)$] For
$D(T)=\{\Phi(G+g): G\mu+g\sigma\in\XX,\,\Lambda(\Phi_\sigma(G))=g)\}\subset L^2(\mu)^b$ and $T=L_V$ on $D(T)$, $T:D(T)\to L^2(\mu)^b$ is a self-adjoint operator.
\item[$(ii)$] If $\Ran(\Lambda)$ is closed, for
$D(T)=\{\Phi(G+g): G\mu+g\sigma\in\XX,\,\Phi_\sigma(G)=\Lambda(g)\}\subset L^2(\mu)^b$ and $T=L_V$ on $D(T)$, $T:D(T)\to L^2(\mu)^b$ is an essentially self-adjoint operator, i.e, $\overline T$ is self-adjoint. Moreover, $D(\overline T)=D(T)+D'$, where $D'$ is the closure in $L^2(\mu)^b$ of $\{{\Phi(h)}:\,h\in\Ker(\Lambda)\}$, and $\overline T(D')=0$.
\item[$(iii)$] For $\Lambda$ and $T$ as in $(ii)$, if $\{{\Phi(h)}:\,h\in\Ker(\Lambda)\}$ is closed, then $T$ is self-adjoint. This occurs, for example, if $\Ker(\Lambda)=\{0\}$.
\end{itemize}
\end{teo}

\begin{remark}\label{r2}
Given $G\in L^2(\mu)^b$, we have $\Phi(G)\in W^{s,2}(\mu)^b$ by Lemma \ref{l2}. On the other hand, given $u\in W^{s,2}(\mu)^b$, if we set $G=L(u)\in L^2(\mu)^b$ (recall that $L$ is of order $s$), we have that $\Phi(G)=\phi*L(u)=u$. Therefore, for $T$ as in Theorem \ref{t1}$(i)$, we obtain
$$D(T)=\big\{u+\Phi(g): u\in W^{s,2}(\mu)^b,\,g\in L^2(\sigma)^b,\,\Lambda(\Tr_\sigma(u))=g\big\},$$
and moreover $T(u+\Phi(g))=L(u)$ for all $u+\Phi(g)\in D(T)$. The respective conclusions hold for $T$ as in Theorem \ref{t1}$(ii)$ and $(iii)$.

\end{remark}

\begin{proof}[Proof of {\em Theorem \ref{t1}}]
We are going to prove $(ii)$ first, which will follow from the following statements:
\begin{itemize}
\item[$(a)$] $D(T)$ is a dense subspace of $ L^2(\mu)^b$,
\item[$(b)$] $T$ is a symmetric operator on $D(T)$,  
\item[$(c)$] $T^*\subset\overline T$.
\end{itemize}

{\em Proof of $(a)$}. That $D(T)$ is a subspace of $ L^2(\mu)^b$ is obvious, so we have to check that it is dense. We know that $\CC^\infty_c(\R^n\setminus\Sigma)^b\subset\DD$ is dense in $L^2(\mu)^b$, because $\sigma$ is $d$-dimensional and $d<n$. Given $F\in\CC^\infty_c(\R^n\setminus\Sigma)^b$ set $G=L(F)\in\CC^\infty_c(\R^n\setminus\Sigma)^b$. Since $\phi$ is the fundamental solution of $L$, we have $\Phi(G)=\phi*G\mu=F$ and $\Phi_\sigma(G)=\Tr_\sigma(F)=0$. Therefore $F=\Phi(G)\in D(T)$, which easily yields $(a)$.

{\em Proof of $(b)$}.
Let $\varphi,\psi\in D(T)$. Then $\varphi=\Phi(G+g)$ and $\psi=\Phi(F+f)$ for some
$G\mu+g\sigma$, $F\mu+f\sigma\in\XX$ with $\Phi_\sigma(G)=\Lambda(g)$ and $\Phi_\sigma(F)=\Lambda(f)$.
By Corollary \ref{c3}, $T(\varphi)=L_V(\varphi)=G$ and $T(\psi)=L_V(\psi)=F$. Hence, using  Lemma \ref{l4} and that $\Lambda$ is self-adjoint in $L^2(\sigma)^b$,
\begin{equation*}
\begin{split}
\langle T(\varphi),\psi\rangle_\mu
-\langle\varphi,T(\psi)\rangle_\mu
&=\langle G,\Phi(F+f)\rangle_\mu
-\langle\Phi(G+g), F\rangle_\mu\\
&=\langle G,\Phi(F)\rangle_\mu
-\langle\Phi(G), F\rangle_\mu
+\langle G,\Phi(f)\rangle_\mu
-\langle\Phi(g), F\rangle_\mu\\
&=\langle\Phi_\sigma(G),f\rangle_\sigma
-\langle g,\Phi_\sigma(F)\rangle_\sigma
=\langle\Lambda(g),f\rangle_\sigma
-\langle g,\Lambda(f)\rangle_\sigma=0,
\end{split}
\end{equation*}
which proves $(b)$.

{\em Proof of $(c)$}. Given an operator
$S:D(S)\subset L^2(\mu)^b\to L^2(\mu)^b$, denote by $\Gamma(S)$ the graph of $S$, i.e.,
$$\Gamma(S)=\{(\varphi,S(\varphi)):\,\varphi\in D(S)\}\subset L^2(\mu)^b\times L^2(\mu)^b.$$

From $(a)$ and $(b)$ we have that $T$ is a densely defined symmetric operator. Thus $T$ is closable by \cite[page 255]{RS}, and $\overline T$ is well defined. Moreover, 
$\Gamma(\overline T)=\overline{\Gamma(T)}$ by \cite[page 250]{RS}. Hence, to prove $(c)$ we only have to verify that $\Gamma(T^*)\subset \overline{\Gamma(T)}$. 

Let $(\psi,F)\in\Gamma(T^*)$, that is, let $\psi, F\in L^2(\mu)^b$ such that
\begin{equation}\label{t1eq1}
\langle T(\varphi),\psi\rangle_\mu=\langle\varphi,F\rangle_\mu
\quad\text{for all }\varphi\in D(T).
\end{equation}
Since $\Lambda$ is bounded, self-adjoint, and $\Ran(\Lambda)$ is closed, then $L^2(\sigma)^b=\Ker(\Lambda)\oplus\Ran(\Lambda)$, so $\Phi_\sigma(F)=h+\Lambda(f)$ for some $h,f\in L^2(\sigma)^b$ with $\Lambda(h)=0$. Notice that $\Phi(h)\in D(T)$ and $T(\Phi(h))=0$, so using (\ref{t1eq1}) with $\varphi=\Phi(h)$ gives $$0=\langle\Phi(h),F\rangle_\mu
=\langle h,\Phi_\sigma(F)\rangle_\sigma
=\langle h,h+\Lambda(f)\rangle_\sigma
=\|h\|_\sigma^2,$$
which actually means that $\Phi_\sigma(F)=\Lambda(f)$. Now, for any $G\mu+g\sigma\in\XX$ such that $\Phi(G+g)\in D(T)$, $(\ref{t1eq1})$ yields
\begin{equation}\label{t1eq4}
\begin{split}
\langle G,\psi\rangle_\mu
&=\langle T(\Phi(G+g)),\psi\rangle_\mu
=\langle\Phi(G+g),F\rangle_\mu
=\langle G,\Phi(F)\rangle_\mu+\langle g,\Phi_\sigma(F)\rangle_\sigma\\
&=\langle G,\Phi(F)\rangle_\mu+\langle g,\Lambda(f)\rangle_\sigma
=\langle G,\Phi(F)\rangle_\mu+\langle \Lambda(g),f\rangle_\sigma\\
&=\langle G,\Phi(F)\rangle_\mu+\langle \Phi_\sigma(G),f\rangle_\sigma
=\langle G,\Phi(F+f)\rangle_\mu,
\end{split}
\end{equation}
which implies that 
\begin{equation}\label{t1eq2}
\langle G,\psi-\Phi(F+f)\rangle_\mu=0
\quad\text{for all }G\in L^2(\mu)^b\text{ such that }\Phi_\sigma(G)\in\Ran(\Lambda).
\end{equation}
Since $\Lambda$ is self-adjoint and $\Ran(\Lambda)$ is closed, $\Phi_\sigma(G)\in\Ran(\Lambda)$ if and only if $0=\langle\Phi_\sigma(G),h\rangle_\sigma
=\langle G,\Phi(h)\rangle_\mu$ for all $h\in\Ker(\Lambda)$. From (\ref{t1eq2}), we deduce that $\langle G,\psi-\Phi(F+f)\rangle_\mu=0$
for all $G\in L^2(\mu)^b$ such that $\langle G,\Phi(h)\rangle_\mu=0$ for all $h\in\Ker(\Lambda)$, that is, 
$$\psi-\Phi(F+f)\in\{{\Phi(h)}:\,h\in\Ker(\Lambda)\}^{\bot\bot}=D',$$
where $D'$ is the closure in $L^2(\mu)^b$ of $\{{\Phi(h)}:\,h\in\Ker(\Lambda)\}$. Hence, there exists $\{h_j\}_{j\in\N}\subset\Ker(\Lambda)$ such that 
\begin{equation}\label{t1eq3}
\psi=\lim_{j\to\infty}\Phi(F+f+h_j)\quad\text{in }L^2(\mu)^b.
\end{equation} 
Set $\psi_j=\Phi(F+f+h_j)$, then $\Phi_\sigma(F)=\Lambda(f)=\Lambda(f+h_j)$ and $T(\psi_j)=F$, so $(\psi_j,F)\in \Gamma(T)$. Moreover, $(\psi,F)=\lim_{j\to\infty}(\psi_j,F)$ in $L^2(\mu)^b\times L^2(\mu)^b$, which implies that 
$(\psi,F)\in\overline{\Gamma(T)}$. Therefore, $\Gamma(T^*)\subset\overline{\Gamma(T)}$, and $(c)$ is proved.

From $(a)$ and $(b)$, $T\subset T^*$. Taking adjoints, $T^*\supset T^{**}$, but $T^{**}=\overline T$ by \cite[page 253]{RS}, so by $(c)$ we have 
$T^*\subset\overline T\subset T^*$, i.e., $\overline T=T^*$. Therefore, 
$(\overline T)^*=T^{**}=\overline T$ and $\overline T$ is self-adjoint, which proves the first statement in $(ii)$. For proving the second one, recall that $\overline T=T^*$, so $\Gamma(\overline T)=\Gamma(T^*)$. In (\ref{t1eq3}) we have seen that any $(\psi,F)\in\Gamma(T^*)$ can be written as $(\Phi(F+f)+\lim_{j\to\infty}\Phi(h_j),F)$ with $h_j\in\Ker(\Lambda)$ and $\Phi_\sigma(F)=\Lambda(f)$, so $D(\overline T)=D(T^*)\subset D(T)+D'$. On the contrary, given $\psi=\Phi(F+f)+\lim_{j\to\infty}\Phi(h_j)\in D(T)+D'$, we have $(\Phi(F+f+h_j),F)\in\Gamma(T)$ for all $j\in\N$, so $(\psi,F)\in\overline{\Gamma(T)}=\Gamma(\overline T)$, which implies that 
$D(\overline T)=D(T)+D'$ and $\overline T(D')=\{0\}$. This finishes the proof of $(ii)$.

For $(iii)$, notice that $\{{\Phi(h)}:\,h\in\Ker(\Lambda)\}\subset D(T)$, so if $\{{\Phi(h)}:\,h\in\Ker(\Lambda)\}=D'$,  then $D'\subset D(T)$ (in particular, $D'$=\{0\} for $\Ker(\Lambda)=\{0\}$), which means that $\overline T=T$. 

Finally, concerning $(i)$, we can proceed as in the proof of $(ii)$. The proof of $(a)$ and $(b)$ are analogous, but instead of $(c)$, we prove that $T^*\subset T$, and then by $(b)$ we conclude that $T=T^*$. Let $(\psi,F)\in\Gamma(T^*)$, that is, let $\psi, F\in L^2(\mu)^b$ such that
$\langle T(\varphi),\psi\rangle_\mu=\langle\varphi,F\rangle_\mu$ for all $\varphi=\Phi(G+g)\in D(T)$. Then, arguing as in (\ref{t1eq4}),
\begin{equation}\label{t1eq5}
\begin{split}
\langle G,\psi\rangle_\mu
&=\langle\Phi(G+g),F\rangle_\mu
=\langle G,\Phi(F)\rangle_\mu+\langle g,\Phi_\sigma(F)\rangle_\sigma\\
&=\langle G,\Phi(F)\rangle_\mu+\langle \Lambda(\Phi_\sigma(G)),\Phi_\sigma(F)\rangle_\sigma
=\langle G,\Phi\big(F+\Lambda(\Phi_\sigma(F))\big)\rangle_\mu.
\end{split}
\end{equation}
Notice that, for any $G\in L^2(\mu)^b$, we have $\Phi\big(G+\Lambda(\Phi_\sigma(G))\big)\in D(T)$, so (\ref{t1eq5}) holds for all $G\in L^2(\mu)^b$. This implies that $\psi=\Phi\big(F+\Lambda(\Phi_\sigma(F))\big)\in D(T)$ and $T^*(\psi)=F=T(\psi)$, so $T^*\subset T$.
\end{proof}

\begin{remark}
Combining the techniques used in the proof of Theorem \ref{t1} one can show that, if $\Lambda:L^2(\sigma)^b\to L^2(\sigma)^b$ is a bounded self-adjoint operator and $\Ran(\Lambda)$ is closed, then for
$$D(T)=\{\Phi(G+\Lambda(g)): G\mu+g\sigma\in\XX,\,\Phi_\sigma(G)-\Lambda^2(g)\in\Ker(\Lambda)\}\subset L^2(\mu)^b$$ and $T=L_V$ on $D(T)$, $T:D(T)\to L^2(\mu)^b$ is a self-adjoint operator. Just notice that, if $\Lambda$ is self-adjoint and $\Ran(\Lambda)$ is closed, any $\Phi_\sigma(G)$ decomposes as $\Lambda(f)+h$ with $h\in\Ker(\Lambda)$, and by using the same decomposition on $f$, we actually have $\Phi_\sigma(G)=\Lambda^2(g)+h$.
Other possible domains $D(T)$ could also be considered. However, the cases stated in Theorem \ref{t1} are enough for our purposes in the next section.
\end{remark}

\section{On the Dirac operator coupled with measure-valued potentials}\label{s2}

This section is devoted to find self-adjoint extenisons of the Dirac operator coupled with measure-valued potentials, where such measures are singular with respect to the Lebesgue measure. Our main tool to obtain such self-adjoint extensions is Theorem \ref{t1}. Throughout this section, we take $n=3$, $b=4$, $s=1$, we denote by $\mu$ the Lebesgue measure in $\R^3$, and
we assume that $\Sigma$ and $\sigma$ are as in any of the cases in Proposition \ref{p1} with $1<d<3$ (that is $0<d<n$ and $2s>n-d$).

Given $m>0$, the free Dirac operator $H:\DD\to\DD$ is defined by 
$H=-i\alpha\cdot\nabla+m\beta,$
where $\alpha=(\alpha_1,\alpha_2,\alpha_3)$,
\begin{equation}\label{s3eq1}
\begin{split}
&\alpha_j
=\left(\begin{array}{cc} 0 & \sigma_j\\
\sigma_j & 0 \end{array}\right) 
\quad\text{for }j=1,2,3,\qquad
\beta
=\left(\begin{array}{cc} I_2 & 0\\
0 & -I_2 \end{array}\right),\quad\text{and}\\
&\sigma_1
=\left(\begin{array}{cc} 0 & 1\\
1 & 0 \end{array}\right), 
\quad\sigma_2
=\left(\begin{array}{cc} 0 & -i\\
i & 0 \end{array}\right), 
\quad\sigma_3
=\left(\begin{array}{cc} 1 & 0\\
0 & -1 \end{array}\right)
\end{split}
\end{equation}
is the family of {\em Pauli matrices}. 

We have $H:\DD\to\DD$ and,  by duality, $H:\DD^*\to\DD^*$. It is very well known that $H$ restricted to $W^{1,2}(\mu)^4$ is a self-adjoint operator (see \cite{Thaller}). We are going to find domains $E$ with $\CC^\infty_c(\R^3\setminus\Sigma)^4\subset E\subset L^2(\mu)^4$ and potentials $V:E\to\DD^*$ such that, for every $\varphi\in E$, $V(\varphi)$ is supported in $\Sigma$, $H_V=H+V$ restricted to $E$ is a self-adjoint operator with respect to $L^2(\mu)^4$, and $H_V=H$ on $\CC^\infty_c(\R^3\setminus\Sigma)^4$. Moreover, when $\sigma$ is the $2$-dimensional surface measure of the boundary of a bounded Lipschitz domain $\Sigma$, we interpret some of such potentials $V$ in terms of the boundary values of the functions in $E$ when we approach to $\Sigma$ from $\R^n\setminus\Sigma$. 

\begin{lema}\label{l7}
The fundamental solution of the symmetric differential operator $H$ is given by 
$$\phi(x)=\frac{e^{-m|x|}}{4\pi|x|}\left(m\beta+(1+m|x|)\,i\alpha\cdot\frac{x}{|x|^2}\right)\quad\text{for }x\in\R^3\setminus\{0\}.$$
Furthermore, $\phi$ satisfies $(i)$, $(ii)$, and $(iii)$ of Section \ref{ss1} for $m>0$.
\end{lema}

\begin{proof}
The relation between the Dirac operator $H$ and the Helmholtz operator $-\Delta+m^2$ is $H^2=(-\Delta+m^2)I_4$. It is well known that the fundamental solution of $-\Delta+m^2$ in $\R^3$ is $\psi_m(x)=e^{-m|x|}(4\pi|x|)^{-1}$ (see \cite[Section 3]{McIntosh} for example). Therefore, by setting 
$$\phi(x)=H(\psi_m(x)I_4)=\frac{e^{-m|x|}}{4\pi|x|}\left(m\beta+(1+m|x|)\,i\alpha\cdot\frac{x}{|x|^2}\right),$$
we deduce that $\phi$ is the fundamental solution of $H$, i.e., $H\phi=\delta_0 I_4$ in the sense of distributions, where $\delta_0$ denotes the Dirac delta measure in $\R^3$ centered at the origin.
Condition $(i)$ of Section \ref{ss1} is trivially satisfied. For the case of $(ii)$, since $\overline{\alpha_j^t}=\alpha_j$ for $j=1,2,3$, and $\overline{\beta^t}=\beta$, we have
\begin{equation*}
\begin{split}
\overline{\phi^t}(y-x)
=\frac{e^{-m|y-x|}}{4\pi|y-x|}\left(m\overline{\beta^t}
+(1+m|y-x|)(-i)\overline{\alpha^t}\cdot\frac{y-x}{|y-x|^2}\right)=\phi(x-y).
\end{split}
\end{equation*}
Conditions $(iii)(a)$ and $(iii)(b)$ are easily verified taking $s=1$ and $\gamma=m$. Finally, for $(iii)(c)$ of Section \ref{ss1}, we know that $H\phi=\delta_0 I_4$ in the sense of distributions. Using the Fourier transform $\FF$ in $\R^3$, 
$$I_4=\FF(\delta_0 I_4)=\FF(H\phi)
=(2\pi\alpha\cdot\xi+m\beta)\FF(\phi)(\xi),$$
hence 
\begin{equation*}
\FF(\phi)(\xi)=(2\pi\alpha\cdot\xi+m\beta)^{-1}
=(4\pi^2|\xi|^2+m^2)^{-1}(2\pi\alpha\cdot\xi+m\beta),
\end{equation*}
which trivially satisfies $(iii)(c)$ of Section \ref{ss1}.
\end{proof}

\begin{coro}
Let $\Sigma$ and $\sigma$ be as in any of the cases in Proposition \ref{p1} with $1<d<3$. Any $\Lambda$ as is Theorem \ref{t1} (with $L=H$) provides a self-adjoint extension of the free Dirac operator $H$ restricted to $\CC^\infty_c(\R^3\setminus\Sigma)^4$.
\end{coro}

\begin{proof}
Lemma \ref{l7} shows that Theorem \ref{t1} and Remark \ref{r2} can be applied to $L=H$.
\end{proof}

\subsection{The case of Lipschitz surfaces}\label{ss5}
For this section, let $\Sigma$ be the boundary of a bounded Lipschitz domain $\Omega_+\subset\R^3$, let $\sigma$ be the surface measure of $\Sigma$, let $N$ denote the outward unit normal vector field on $\Sigma$ with respect to $\Omega_+$, and set $\Omega_-=\R^3\setminus\overline{\Omega_+}$, so $\Sigma=\partial\Omega_+=\partial\Omega_-$. We keep the notation introduced in Section \ref{s1}, but with $L=H$ and $\phi$ given by Lemma \ref{l7}.

The following lemma is somehow contained in \cite[Theorem 4.4]{McIntosh}, but we state and prove it here for the sake of completeness. We are grateful to Luis Escauriaza for showing us a simple argument to prove Lemma \ref{l5}$(ii)$.

\begin{lema}\label{l5}
Given $g\in L^2(\sigma)^4$ and $x\in\Sigma$, set 
\begin{equation*}
\begin{split}
C_\sigma (g)(x)=\lim_{\epsilon\searrow0}\int_{|x-z|>\epsilon}\phi(x-z)g(z)\,d\sigma(z) 
\quad\text{and}\quad
C_{\pm}(g)(x)=\lim_{\Omega_{\pm}\ni y\stackrel{nt}{\longrightarrow} x}
\Phi(g)(y),
\end{split}
\end{equation*}
where $\textstyle{\Omega_{\pm}\ni y\stackrel{nt}{\longrightarrow} x}$ means that $y\in\Omega_{\pm}$ tends to $x\in\Sigma$ non-tangentially. Then $C_\sigma (g)(x)$ and $C_{\pm}(g)(x)$ exist for $\sigma$-a.e. $x\in\Sigma$, and $C_\sigma,C_\pm: L^2(\sigma)^4\to L^2(\sigma)^4$ are linear bounded operators. Moreover, the following  holds:
\begin{itemize}
\item[$(i)$] $C_\pm =\mp\frac{i}{2}\,(\alpha\cdot N)+C_\sigma$ (Plemelj--Sokhotski jump formulae),
\item[$(ii)$] $-4(C_\sigma(\alpha\cdot N))^2=I_4$.
\end{itemize}
\end{lema}

\begin{proof}
The first statements of the lemma and $(i)$ are a consequence of the following well known fact (see \cite{Kellogg}, or \cite[page 1071]{Escauriaza} for example). Given $f\in L^2(\sigma)$, then for $\sigma$-a.e. $x=(x_1,x_2,x_3)\in\Sigma$ and for all $j=1,2,3$, 
\begin{equation}\label{l5eq1}
\begin{split}
\lim_{\Omega_{\pm}\ni y\stackrel{nt}{\longrightarrow} x}
\int\frac{y_j-z_j}{4\pi|y-z|^3}\,f(z)\,d\sigma(z)=
\mp\frac{f(x)}{2}\,N_j(x)+
\lim_{\epsilon\searrow0}\int_{|x-z|>\epsilon}\frac{x_j-z_j}{4\pi|x-z|^3}\,f(z)\,d\sigma(z),
\end{split}
\end{equation}
and the integrals in (\ref{l5eq1}) define linear operators which are bounded in $L^2(\sigma)$.

We write
\begin{equation}\label{l5eq2}
\begin{split}
\phi(x)&=\frac{e^{-m|x|}}{4\pi|x|}\,m\left(\beta+i\alpha\cdot\frac{x}{|x|}\right)+\frac{e^{-m|x|}-1}{4\pi}\,i\left(\alpha\cdot\frac{x}{|x|^3}\right)
+\frac{i}{4\pi}\left(\alpha\cdot\frac{x}{|x|^3}\right)\\
&=\omega_1(x)+\omega_2(x)+\omega_3(x).
\end{split}
\end{equation}
For $j=1,2$ and any $1\leq k,l\leq4$, we have 
$|(\omega_j)_{k,l}(x)|=O(|x|^{-1})$ for $|x|\to0$. Using that $\sigma$ is $2$-dimensional and rather standard arguments (essentially, using that $\Sigma$ is bounded, the generalized Young's inequality, and the dominate convergence theorem), it is not hard to show that, for $j=1,2$, 
\begin{equation}\label{l5eq3}
\begin{split}
\lim_{\Omega_{\pm}\ni y\stackrel{nt}{\longrightarrow} x}
\int\omega_j(y-z)g(z)\,d\sigma(z)=
\lim_{\epsilon\searrow0}\int_{|x-z|>\epsilon}\omega_j(x-z)
g(z)\,d\sigma(z)
\end{split}
\end{equation}
for all $g\in L^2(\sigma)^4$ and $\sigma$-a.e. $x\in\Sigma$, and the integrals in (\ref{l5eq3}) define linear operators which are bounded in $L^2(\sigma)^4$.
For the case of $\omega_3$, using (\ref{l5eq1}) we obtain
\begin{equation}\label{l5eq4}
\begin{split}
\lim_{\Omega_{\pm}\ni y\stackrel{nt}{\longrightarrow} x}
&\int\omega_3(y-z)g(z)\,d\sigma(z)
=\lim_{\Omega_{\pm}\ni y\stackrel{nt}{\longrightarrow} x}
i\sum_{j=1}^3\int\frac{y_j-z_j}{4\pi|y-z|^3}\,\alpha_jg(z)\,d\sigma(z)\\
&=i\sum_{j=1}^3\left(\mp\frac{1}{2}\,\alpha_j g(x)N_j(x)+\lim_{\epsilon\searrow0}\int_{|x-z|>\epsilon}\frac{x_j-z_j}{4\pi|x-z|^3}\,\alpha_jg(z)\,d\sigma(z)\right)\\
&=\mp\frac{i}{2}\,(\alpha\cdot N(x)) g(x)+\lim_{\epsilon\searrow0}\int_{|x-z|>\epsilon}\omega_3(x-z)g(z)\,d\sigma(z).
\end{split}
\end{equation}
Then $(i)$ follows by (\ref{l5eq2}), (\ref{l5eq3}), and (\ref{l5eq4}).

In order to prove $(ii)$, recall the following reproducing formula (see \cite[Section 3]{McIntosh}, for example): if $\Omega$ is a bounded Lipschitz domain in $\R^3$ and $f\in\CC^\infty(\Omega)^4$ satisfies $H(f)=0$ in $\Omega$ and has non-tangential boundary values in $L^2(\sigma_\Omega)^4$, then
\begin{equation}\label{l5eq5}
\begin{split}
f(x)=\int_{\partial\Omega}\phi(x-z)(i\alpha\cdot N_\Omega(z))f(z)\,d\sigma_\Omega(z)
\end{split}
\end{equation}
for all $x\in\Omega$, where $N_\Omega$ and $\sigma_\Omega$ are the outward unit normal vector field and surface measure of $\partial\Omega$ respectively. This reproducing formula can be proved using integration by parts on 
\begin{equation*}
\begin{split}
\int_{\Omega\setminus B(x,\epsilon)}H(f)(z)\cdot\overline{\phi(z-x)e_j}\,d\mu(z)\quad\text{for }j=1,2,3,4,
\end{split}
\end{equation*}
and taking $\epsilon\searrow0$, where $e_1=(1,0,0,0),\ldots,e_4=(0,0,0,1)$, and $B(x,\epsilon)$ is the ball centered at $x$ and with radius $\epsilon>0$. 

Let $g\in L^2(\sigma)^4$. Since $H(\Phi((i\alpha\cdot N)g))=0$ in $\Omega_+$, using (\ref{l5eq5}) we have that, for all $x\in\Omega_+$,
\begin{equation}\label{l5eq6}
\begin{split}
\Phi((i\alpha\cdot N)g)(x)
=\Phi\big((i\alpha\cdot N)C_+((i\alpha\cdot N)g)\big)(x).
\end{split}
\end{equation}
By approaching to $\Sigma$ non-tangentially, we deduce from $(i)$ and (\ref{l5eq6}) that 
\begin{equation*}
\begin{split}
\frac{1}{2}\,g+C_\sigma ((i\alpha\cdot N)g)
&=C_+((i\alpha\cdot N)g)
=C_+\big((i\alpha\cdot N)C_+((i\alpha\cdot N)g)\big)\\
&=\frac{1}{2}\,\Big(\frac{1}{2}\,g+C_\sigma ((i\alpha\cdot N)g)\Big)+C_\sigma \Big((i\alpha\cdot N)\Big(\frac{1}{2}\,g+C_\sigma ((i\alpha\cdot N)g)\Big)\Big)\\
&=\frac{1}{4}\,g+C_\sigma((i\alpha\cdot N)g)
-(C_\sigma \big(\alpha\cdot N))^2(g),
\end{split}
\end{equation*}
which proves $(ii)$. Let us mention that, if one argues with $\Omega_-$ and $C_-$ instead of $\Omega_+$ and $C_+$, one obtains the same result. The lemma is finally proved.
\end{proof}

\begin{remark}\label{r3}
Let $\varphi=\Phi(G+g)$ for some $G\mu+g\sigma\in\XX$, and set $\varphi_\pm=\Phi_\sigma(G)+C_\pm(g)$. Since $V(\varphi)=-g\sigma$ by definition (see Corollary \ref{c3}), Lemma \ref{l5}$(i)$ yields
$$V(\varphi)=-i(\alpha\cdot N)(\varphi_+-\varphi_-)\sigma.$$
This is consistent with the fact that, if $\varphi$ is a function which is smooth in $\Sigma^c$ and has a jump at $\Sigma$, then $H(\varphi)=\chi_{\Sigma^c}H(\varphi)\mu-i(\alpha\cdot N)(\varphi_- -\varphi_+)\sigma$ distributionally (this is an easy exercise left for the reader), so that for having $(H+V)(\varphi)\in L^2(\mu)^4$ one needs to take $V(\varphi)=i(\alpha\cdot N)(\varphi_--\varphi_+)\sigma$.
\end{remark}

\begin{lema}\label{l8}
If $\Sigma$ is $\CC^2$, the anticommutator $\{\alpha\cdot N,C_\sigma\}=(\alpha\cdot N)C_\sigma+C_\sigma(\alpha\cdot N)$ is a compact operator on $L^2(\sigma)^4$.
\end{lema}

\begin{proof}
Given $x\in\Sigma$ and $y\in\R^3$, a simple computation shows that 
\begin{equation}\label{l8eq2}
(\alpha\cdot N(x))(\alpha\cdot y)
=-(\alpha\cdot y)(\alpha\cdot N(x))+2(N(x)\cdot y)I_4.
\end{equation}
Since the $\alpha_j$'s anticommute with $\beta$, 
(\ref{l8eq2}) yields 
$$(\alpha\cdot N(x))\phi(y)
=-\phi(y)(\alpha\cdot N(x))
+i(2\pi)^{-1}e^{-m|y|}|y|^{-3}(1+m|y|)(N(x)\cdot y)I_4.$$ 
Therefore, for $g\in L^2(\sigma)^4$, 
$\{\alpha\cdot N,C_\sigma\}(g)(x)
=\lim_{\epsilon\searrow0}\int_{|x-z|>\epsilon}
K(x,z)g(z)\,d\sigma(z)$, where
\begin{equation}\label{l8eq1}
\begin{split}
K(x,z)=\phi(x-z)(\alpha\cdot (N(z)-N(x))+
\frac{ie^{-m|x-z|}}{2\pi|x-z|^{3}}\,(1+m|x-z|)(N(x)\cdot(x-z))I_4.
\end{split}
\end{equation}
Since $\Sigma$ is $\CC^2$, it is not hard to show that
$\sup_{1\leq j,k\leq 4}|K_{j,k}(x,z)|=O(|x-z|^{-1})$ when $|x-z|$ tends to zero, because $|N(x)-N(z)|=O(|x-z|)$ and 
$|N(x)\cdot(x-z)|=O(|x-z|^2)$ for $x,z\in\Sigma$ with $|x-z|$ small enough (see \cite[Lemma 3.15]{Folland}, for example). Using this estimate, one can easily adapt the proof of \cite[Proposition 3.11]{Folland} to show that $\{\alpha\cdot N,C_\sigma\}$ is a compact operator.
\end{proof}

\begin{remark}
The $\CC^2$ condition on $\Sigma$ is not sharp, but it is enough for our purposes. One can require less regularity on $\Sigma$ and still obtain compactness of the anticommutator. For example, if $\Sigma$ is $\CC^1$, the methods developed in \cite{Fabes} would work.
\end{remark}

\begin{lema}\label{l9}
Given $\lambda\in\R\setminus\{0\}$, set $\Lambda_{\pm}=1/\lambda \pm C_\sigma$. Then, $\Lambda_{\pm}: L^2(\sigma)^4\to L^2(\sigma)^4$ are linear bounded self-adjoint operators. Moreover, if $\Sigma$ is $\CC^2$ and $\lambda\in\R\setminus\{-2,0,2\}$ then $\Ran(\Lambda_{\pm})$ are closed.
\end{lema}

\begin{proof}
That $\Lambda_\pm$ are bounded and self-adjoint follow essentially by Lemma \ref{l5} and $(ii)$ in Section \ref{ss1}, we omit the details. It remains to check that $\Ran(\Lambda_{+})$ is closed when $\Sigma$ is a $\CC^2$ surface and $\lambda\in\R\setminus\{-2,0,2\}$, the proof for $\Ran(\Lambda_{-})$ is analogous. 

Recall that $(C_\sigma(\alpha\cdot N))^2=-1/4$ by Lemma \ref{l5}$(ii)$, and $(\alpha\cdot N)^2=I_4$, so
\begin{equation}\label{l9eq1}
\begin{split}
\Lambda_+\Lambda_-
&=\Lambda_-\Lambda_+
=1/\lambda^2-C_\sigma^2
=1/\lambda^2-1/4-C_\sigma(\alpha\cdot N)\{\alpha\cdot N,C_\sigma\}=a-K,
\end{split}
\end{equation}
where $a=1/\lambda^2-1/4$ and $K=C_\sigma(\alpha\cdot N)\{\alpha\cdot N,C_\sigma\}$. Since $C_\sigma(\alpha\cdot N)$ is bounded, $K$ is a compact operator by Lemma \ref{l8}, thus $\Ran(a-K)$ is closed for all $a\in\R\setminus\{0\}$ (i.e, for all $\lambda\in\R\setminus\{-2,0,2\}$) by Fredholm's theorem (see, \cite[Theorem 0.38$(c)$]{Folland}, for example). Furthermore, (\ref{l9eq1}) shows that $K$ is self-adjoint, since $\Lambda_\pm$ are self-adjoint and commute.

Given $f\in\overline{\Ran(\Lambda_+)}$, there exist $g_j\in L^2(\sigma)^4$ with $j\in\N$ such that $f=\lim_ {j\to\infty}\Lambda_+(g_j)$. Then, for any $h\in\Ker(\Lambda_+\Lambda_-)$,  
\begin{equation*}
\begin{split}
\langle \Lambda_-(f),h\rangle_\sigma
=\langle f,\Lambda_-(h)\rangle_\sigma
=\lim_ {j\to\infty}\langle \Lambda_+(g_j),\Lambda_-(h)\rangle_\sigma
=\lim_ {j\to\infty}\langle g_j,\Lambda_+\Lambda_-(h)\rangle_\sigma=0,
\end{split}
\end{equation*}
thus $\Lambda_-(f)\in\Ker(\Lambda_+\Lambda_-)^\bot$. Using (\ref{l9eq1}) and that $a-K$ has closed range for all $a\neq0$, we have $\Ker(\Lambda_+\Lambda_-)^\bot
=\overline{\Ran(\Lambda_-\Lambda_+)}=\Ran(a-K)$, so there exists 
$g\in L^2(\sigma)^4$ such that $\Lambda_-(f)=(a-K)g=\Lambda_-\Lambda_+(g)$, which yields $f-\Lambda_+(g)\in\Ker(\Lambda_-)$. Notice that $\Lambda_++\Lambda_-=2/\lambda$, hence
$2\lambda^{-1}\,(f-\Lambda_+(g))=\Lambda_+(f-\Lambda_+(g))$, which implies that 
$$f=\Lambda_+\left(g+\frac{\lambda}{2}\,(f-\Lambda_+(g))\right)\in\Ran(\Lambda_+).$$
Therefore, $\Ran(\Lambda_+)$ is closed and the lemma is proved.
\end{proof}

\begin{teo}\label{t2}
Assume that $\Sigma$ is $\CC^2$. Given $\lambda\in\R$, let $T$ be the operator defined by
$$D(T)=\big\{u+\Phi(g): u\in W^{1,2}(\mu)^4,\,g\in L^2(\sigma)^4,\,\lambda\Tr_\sigma(u)=-(1+\lambda C_\sigma)(g)\big\}$$
and $T=H+V_\lambda$ on $D(T)$, where 
$$V_\lambda(\varphi)=\frac{\lambda}{2}(\varphi_++\varphi_-)\sigma$$ and 
$\varphi_\pm=\Tr_\sigma(u)+C_\pm (g)$ for $\varphi=u+\Phi(g)\in D(T)$. 
If $\lambda\neq\pm2$, then $T:D(T)\subset L^2(\mu)^4\to L^2(\mu)^4$ is self-adjoint. 

There exists a finite or countable sequence $\{\lambda_j\}_{j\in J}\subset(0,\infty)$ depending only on $\sigma$ and $m>0$, and whose unique possible accumulation point is $2$, such that the following holds:
\begin{itemize}
\item[$(i)$] If $|\lambda|\neq\lambda_j$ for all $j\in J$ and $\varphi\in D(T)$ is such that $T(\varphi)=0$, then $\varphi=0$.
\item[$(ii)$]  If $|\lambda|=\lambda_j$ for some $j\in J$, there exist a non-trivial $\varphi=\Phi(g)$ with $g\in L^2(\sigma)^4$ such that either $$(H+V_\lambda)(\varphi)=0\quad\text{or}\quad(H+V_{-\lambda})(\varphi)=0.$$
In particular, if $\sigma=s_\#\sigma$ then there exists a non-trivial $\varphi\in D(T)$ such that $T(\varphi)=0$, where $s(x)=-x$ for $x\in\R^3$ and $s_\#\sigma$ is the image measure of $\sigma$ with respect to $s$.
\end{itemize}
\end{teo}

\begin{proof}
We are going to prove first that $T$ is self-adjoint for all $\lambda\neq\pm2$.
If $\lambda=0$ then $D(T)=W^{1,2}(\mu)^4$ and $V_\lambda=0$, so we recover the classical self-adjointness of the free Dirac operator $H$, for example by applying Theorem \ref{t1}$(i)$ and Remark \ref{r2} with $L=H$ and $\Lambda=0$. Hence, $T$ is self-adjoint for $\lambda=0$.

Assume that $\lambda\neq0$, and set $\Lambda=-(1/\lambda+C_\sigma)$. Then, using Remark \ref{r2}, we have 
\begin{equation}\label{t2eq2}
D(T)=\big\{\Phi(G+g): G\mu+g\sigma\in\XX,\,\Phi_\sigma(G)=\Lambda (g)\big\}\subset L^2(\mu)^4.
\end{equation}
Let $\varphi=\Phi(G+g)\in D(T)$, then 
$\varphi_\pm=\Phi_\sigma(G)\mp\frac{i}{2}\,(\alpha\cdot N)g
+C_\sigma(g)$ by Lemma \ref{l5}$(i)$, so 
\begin{equation}\label{t2eq1}
V_\lambda(\varphi)
=\lambda(\Phi_\sigma(G)+C_\sigma(g))\sigma
=\lambda(\Lambda(g)+C_\sigma(g))\sigma
=-g\sigma,
\end{equation}
thus $V_\lambda$ restricted to $D(T)$ coincides with the potential $V$ introduced in Corollary \ref{c3}, and $T:D(T)\to L^2(\mu)^4$. Moreover, if $\lambda\neq\pm2$ then $\Lambda$ is a linear bounded self-adjoint operator with $\Ran(\Lambda)$ closed, by Lemma \ref{l9}.
Arguing as in (\ref{l9eq1}), if we set $\Lambda_+=-\Lambda$ and $\Lambda_-=1/\lambda-C_\sigma$, we have 
$$\Lambda_+\Lambda_-=\Lambda_-\Lambda_+=a-K$$ 
with $a=1/\lambda^2-1/4$ and $K=C_\sigma(\alpha\cdot N)\{\alpha\cdot N,C_\sigma\}$. We already know from the proof of Lemma \ref{l9} that $K$ is bounded, compact, and self-adjoint, thus the eigenvalues of $K$ form a finite or countable bounded sequence $\{a_j\}_{j\in J'}\subset\R$ whose unique possible accumulation point is $0$, by Fredholm's Theorem (see \cite[Theorem 0.38$(a)$]{Folland}, for example). Furthermore, \cite[Theorem 0.38$(a)$]{Folland} also gives that $\Ker(a-K)$ has finite dimension for all $a\neq0$, that is for all $\lambda\neq\pm2$. Since $-a+K=\Lambda_-\Lambda$, then $\Ker(\Lambda)$ must be finite dimensional, and this easily implies that $\{{\Phi(h)}:\,h\in\Ker(\Lambda)\}$ is closed in $L^2(\mu)^4$. Therefore, Theorem \ref{t1}$(iii)$ shows that $T$ is self-adjoint for all $\lambda\neq -2, 0,2$.

In order to prove the second part of the theorem, take $\lambda_j=2(1+4a_j)^{-1/2}$ whenever $a_j>-1/4$, and notice that 
$\{\lambda_j\}_{j\in J}$ can only accumulate at $2$. Concernig $(i)$, assume that $\varphi=u+\Phi(g)\in D(T)$ is such that $T(\varphi)=0$ (notice that $g=0$ if $\lambda=0$, by the definition of $D(T)$ in the statement of the theorem). By Remark \ref{r2}, we may assume that $u=\Phi(G)$ for some $G\in L^2(\mu)^4$, and (\ref{t2eq1}) yields $$0=T(\varphi)=(H+V)(\Phi(G+g))=G,$$ so actually $\varphi=\Phi(g)$ (and we are done if $\lambda=0$). From the choice of $\lambda_j$, we already know that if $|\lambda|\neq\lambda_j$ for all $j\in J$ then $a\not\in\{a_j\}_{j\in J}$, and hence $\Ker(\Lambda)\subset\Ker(-a+K)=\{0\}$. Since $\varphi\in D(T)$, by (\ref{t2eq2}) we must have $0=\Phi_\sigma(G)=\Lambda(g)$, and since $\Lambda$ is injective, we conclude that $g=0$. This proves of $(i)$.

Let us now prove the first part of $(ii)$. If $|\lambda|=\lambda_j$ for some $j\in J$ (in particular, $\lambda\neq0$) then $a\in\{a_j\}_{j\in J}$, so $a$ is an eigenvalue of $K$ and we can pick $0\neq f\in L^2(\sigma)^4$ such that 
\begin{equation}\label{t2eq3}
(\Lambda_+\Lambda_-)(f)=(\Lambda_-\Lambda_+)(f)=(-a+K)(f)=0.
\end{equation}
Recall that $\Lambda_++\Lambda_-=2/\lambda$, which means that either $\Lambda_+(f)\neq0$ or $\Lambda_-(f)\neq0$. If $\Lambda_-(f)\neq0$, by setting $g=\Lambda_-(f)$, (\ref{t2eq3}) gives $\Lambda(g)=-\Lambda_+(g)=0$. Using (\ref{t2eq2}), we have $\Phi(g)\in D(T)$ and, moreover, $T(\Phi(g))=(H+V)(\Phi(g))=0$, so we are done. Assuming now that $\Lambda_+(f)\neq0$, set $g=\Lambda_+(f)$ and $\varphi=\Phi(g)$. Then $0=(\Lambda_-\Lambda_+)(f)=\Lambda_-(g)=(1/\lambda-C_\sigma)(g)$ by (\ref{t2eq3}), so
\begin{equation*}
V(\varphi)=-g\sigma=-\lambda C_\sigma(g)\sigma
=-\frac{\lambda}{2}(\varphi_++\varphi_-)\sigma=V_{-\lambda}(\varphi),
\end{equation*}
and therefore $(H+V_{-\lambda})(\varphi)=0$. 

Finally, we are going to prove the last statement of $(ii)$, so we assume that $\sigma=s_\#\sigma$. As we have already seen, if $\Lambda_-(f)\neq0$ then we can find a non-trivial $\varphi\in D(T)$ such that $T(\varphi)=0$. So assume now that $\Lambda_-(f)=0$ and set $g=-\tau\Lambda_+(f)\circ s$, where 
$$\tau
=\left(\begin{array}{cc} 0 & I_2\\
I_2 & 0 \end{array}\right).$$
Notice that $g\neq0$ because $\Lambda_+(f)\neq0$ and, since $\sigma=s_\#\sigma$, we have $g\in L^2(\sigma)^4$.
It is straightforward to check that 
$-\phi(z)\tau=\tau\phi(-z)$ for all $z\in\R^3\setminus\{0\}$. Therefore,
\begin{equation}\label{t2eq5}
\begin{split}
C_\sigma(g)(x)&=\lim_{\epsilon\searrow0}\int_{|x-y|>\epsilon}-\phi(x-y)\tau \Lambda_+(f)(-y)\,d\sigma(y)\\
&=\tau\lim_{\epsilon\searrow0}\int_{|x-y|>\epsilon}\phi(-x+y)\Lambda_+(f)(-y)\,ds_\#\sigma(y)\\
&=\tau\lim_{\epsilon\searrow0}\int_{|x+y|>\epsilon}\phi(-x-y) \Lambda_+(f)(y)\,d\sigma(y)
=\tau C_\sigma(\Lambda_+(f))(-x).
\end{split}
\end{equation}
Recall from (\ref{t2eq3}) that $(\Lambda_-\Lambda_+)(f)=0$, so
\begin{equation}\label{t2eq6}
\begin{split}
\tau C_\sigma(\Lambda_+(f))
&=-\tau(\lambda^{-1}-C_\sigma)(\Lambda_+(f))
+\lambda^{-1}\tau\Lambda_+(f)\\
&=-\tau(\Lambda_-\Lambda_+)(f)+\lambda^{-1}\tau\Lambda_+(f)
=\lambda^{-1}\tau\Lambda_+(f).
\end{split}
\end{equation}
By (\ref{t2eq5}) and (\ref{t2eq6}), we have $C_\sigma(g)=-\lambda^{-1}g$, 
which means that $\Lambda(g)=-(1/\lambda+C_\sigma)(g)=0$. Hence $\Phi(g)\in D(T)$, and $T(\Phi(g))=0$ by (\ref{t2eq1}). The theorem is finally proved.
\end{proof}

\begin{remark}
Despite the domain $D(T)$ appearing in Theorem \ref{t2} {\em a priori} depends on $m>0$ (since it is defined in terms of $\phi$), a straightforward application of the Kato-Rellich theorem to the self-adjoint operator $H+V_\lambda$ given by Theorem \ref{t2} and the symmetric bounded operator $m_0\beta$ (for any given $m_0>0$) shows that actually $D(T)$ is independent of $m$ (see \cite[Theorem X.12]{RS}, for example). This could also be verified directly on the domain by working with the operator $C_\sigma$.
\end{remark}

The next proposition contains some particularities concerning Theorem \ref{t2} in the case that $\Sigma$ is a plane or a sphere.

\begin{propo}\label{p2}
Let $T$ be as in Theorem \ref{t2}.
\begin{itemize} 
\item[$(i)$] Assume that $\Sigma=\R^2\times\{0\}\subset\R^3$. Then the following hold:
\begin{itemize}
\item[$(a)$] If $\lambda\neq\pm2$ then $T$ is self-adjoint and, if $\varphi\in D(T)$ satisfies $T(\varphi)=0$, then $\varphi=0$.
\item[$(b)$] If $\lambda=\pm2$ then $T$ is essentially self-adjoint.  Moreover, $D(\overline T)=D(T)+\Phi(X)$ and $\overline T(\Phi(X))=0$, where $X$ is the completion of $\ker(1/\lambda+C_\sigma)$ with respect to the norm
$$\|h\|^2=\langle|S|^{-1}(h),h\rangle_\sigma,\quad\text{where}\quad S=\alpha_3(\alpha_1\partial_{x_1}+\alpha_2\partial_{x_2}
+im\beta).$$
In particular $D(T)\subsetneq D(\overline T)$.
In addition, for $\lambda=\pm2$ there exists a non-trivial $\varphi\in D(T)$ such that $T(\varphi)=0$.
\end{itemize}
\item[$(ii)$] Assume that $\Sigma=\{x\in\R^3:\, |x|=1\}\subset\R^3$. Then 
there exists some $\lambda_j\neq2$, where $\{\lambda_j\}_{j\in J}$ is the sequence given by Theorem \ref{t2}. In particular, $\{\alpha\cdot N,C_\sigma\}$ is not identically zero.
\end{itemize}
\end{propo}

\begin{proof}
We keep the notation used in the proof Theorem \ref{t2}. Concerning $(i)(a)$, we can not apply Theorem \ref{t2} directly because $\Sigma$ is unbounded (and hence $K$ might lose its compacity). However, recall that $K=C_\sigma(\alpha\cdot N)\{\alpha\cdot N,C_\sigma\}$ and that the kernel of $\{\alpha\cdot N,C_\sigma\}$ is given by (\ref{l8eq1}). Since $\Sigma=\R^2\times\{0\}$, then $N$ is constant, so $N(z)-N(x)=0$ for all $x,z\in\Sigma$ and, similarly, $N(x)\cdot(x-z)=0$. This implies by (\ref{l8eq1}) that the kernel defining $\{\alpha\cdot N,C_\sigma\}$ is identically zero, so $K=0$ and $\Lambda_+\Lambda_-=1/4-1/\lambda^2$. Therefore, $\Lambda=-\Lambda_+$ is invertible for all $\lambda\neq\pm2$, so $\Ran(\Lambda)$ is closed and $\Ker(\Lambda)=0$. Then, Theorem \ref{t1}$(iii)$ in combination with (\ref{t2eq2}) and (\ref{t2eq1}) shows that $T$ si self-adjoint for all $\lambda\neq\pm2$ (Remark \ref{r1} justifies the use of (\ref{t2eq1})). The second statement of $(i)(a)$ follows essentially as Theorem \ref{t2}$(i)$, we leave the details for the reader.

In order to prove the first statement of $(i)(b)$, assume for example that $\lambda=2$. Then, we have $\Lambda_++\Lambda_-=I_4$,
$\Lambda_+\Lambda_-=\Lambda_-\Lambda_+=0$ and, as a consequence, $\Lambda_+^2=\Lambda_+$ and $\Lambda_-^2=\Lambda_-$. Thus $\Lambda_\pm$ are self-adjoint projections in $L^2(\sigma)^4$, and hence they have closed range by \cite[Theorem 12.14$(c)$]{Rudin}. Therefore, $\Ran(\Lambda)$ is closed for $\lambda=2$ and Theorem \ref{t1}$(ii)$ applies, showing that $T$ is essentially self-adjoint. The case $\lambda=-2$ follows by similar arguments.

We are going to prove the second statement of $(i)(b)$, so assume that $\lambda=2$ (the case $\lambda=-2$ is similar). We take $N(x)=(0,0,-1)$ for all $x\in\Sigma$, that is, $\Omega_+=\R^2\times(0,\infty)$. 
Notice that $S=\alpha_3(\alpha_1\partial_{x_1}+\alpha_2\partial_{x_2}
+im\beta)$ only acts on the coordinates $x_1$ and $x_2$.
Using the Fourier transform on the $x_1$ and $x_2$ variables together with the anticommutation properties of $\beta$ and the $\alpha_j$'s, it is easy to show that $S$ is a self-adjoint operator on $W^{1,2}(dx_1dx_2)$, so  its eigenvalues are real. Let $P_\pm$ be the positive/negative projection operators associated to $S$, i.e., given a function $f(x_1,x_2)$ decomposed in terms of the eigenvectors of $S$, $P_+(f)$ corresponds to the part of the decomposition of $f$ relative to the eigenfunctions with positive eigenvalue, and $P_-(f)$ corresponds to the negative ones. In particular, $SP_+\geq0$ and $SP_-\leq0$. Define the positive operator 
$|S|=SP_+-SP_-$ and let $|S|^{-1/2}$ be the positive square root of the inverse of $|S|$, which exists because of the invertibility and positivity of $|S|$ (use the Fourier transform). 

For any given function $\varphi$, $H(\varphi)=0$ in $\R^2\setminus\Sigma$ is equivalent to $\partial_{x_3}\varphi=-S(\varphi)$ for all $x_3\neq0$. It is an exercise to show that $\partial_{x_3}\varphi=-S(\varphi)$ if and only if 
$\partial_{x_3}\big(|S|^{-1/2}(\varphi)\big)=-S|S|^{-1/2}(\varphi)$.
Let $\varphi=\Phi(h)\in\{\Phi(g):\,g\in\Ker(\Lambda)\}$, and set 
$\psi=|S|^{-1/2}(\varphi)$. Then, since $H(\varphi)=0$ in $\R^2\setminus\Sigma$, by the previous comments we have 
\begin{equation}\label{p2eq1b}
\partial_{x_3}\psi=-S(\psi)\quad\text{for all }x_3\neq0.
\end{equation}
Moreover, it is not hard to show that $\psi$ has non-tangential boundary values at $\Sigma$ from $\Omega_\pm$ and, actually, $\psi_\pm=|S|^{-1/2}\varphi_\pm$.
By standard arguments, this implies that
\begin{equation}\label{p2eq4}
\psi(x_1,x_2,x_3)=\left\{
\begin{split}
&e^{-x_3S}\psi_+(x_1,x_2)\quad\text{for all }x_3>0,\\
&e^{x_3S}\psi_-(x_1,x_2)\quad\text{for all }x_3<0,
\end{split}
\right.
\end{equation}
and $P_\pm(\psi_\mp)=0$. If we multiply (\ref{p2eq1b}) by $2\overline\psi$ and we take real parts, we obtain
$\partial_{x_3}(|\psi|^2)=-2\Re\left(S(\psi)\cdot\overline\varphi\right)$, and then integrating in $\Omega_\pm$ and using (\ref{p2eq4}), we deduce
\begin{equation}\label{p2eq5}
\begin{split}
\pm\int_{\Sigma}|\psi_\pm|^2\,d\sigma
=2\int_{\Omega_\pm}\Re\Big(S(\psi)\cdot\overline\psi\Big)\,d\mu,
=\pm2\int_{\Omega_\pm}|S|(\psi)\cdot\overline\psi\,d\mu.
\end{split}
\end{equation}

Recall that $h\in\Ker(\Lambda)$, so $\Lambda(h)=-(1/2+C_\sigma)(h)=0$. Hence
$\varphi_\pm=\frac{1}{2}(\pm i\alpha_3-I_4)h$ by Lemma \ref{l5}$(i)$, and so 
\begin{equation}\label{p2eq3}
\varphi_++\varphi_-=-h\quad\text{and}
\quad\psi_++\psi_-=-|S|^{-1/2}(h).
\end{equation}
Notice that, since $P_\pm(\psi_\mp)=0$ and $P_\pm$ are complementary projections, $\psi_+$ and $\psi_-$ are orthogonal and thus 
$\|\psi_++\psi_-\|^2_\sigma
=\|\psi_+\|^2_\sigma+\|\psi_-\|^2_\sigma$. Therefore, by (\ref{p2eq3}) and (\ref{p2eq5}), 
\begin{equation}\label{p2eq2}
\begin{split}
\langle|S|^{-1}(h),h\rangle_\sigma
&=\|\psi_++\psi_-\|^2_\sigma
=\|\psi_+\|^2_\sigma+\|\psi_-\|^2_\sigma
=2\langle|S|(\psi),\psi\rangle_\mu
=2\|\varphi\|^2_\mu,
\end{split}
\end{equation}
since we have set $\psi=|S|^{-1/2}(\varphi)$.
Therefore, looking at (\ref{p2eq2}), we deduce that the closure in $L^2(\mu)^4$ of $\{\Phi(h):\,h\in\Ker(\Lambda)\}$, which we denote by $D'$, corresponds to the image by $\Phi$ of the completion of $\ker(\Lambda)\subset L^2(\sigma)^4$ with respect to the norm given by the left hand side of $(\ref{p2eq2})$. This completion of $\ker(\Lambda)$ is not contained in $L^2(\sigma)^4$ because, roughly speaking, $\ker(\Lambda)$ is big enough. Indeed, on the Fourier side, it is not hard to show that
$$\FF(\Lambda)(\xi_1,\xi_2)=-\frac{1}{2}-\FF(C_\sigma)(\xi_1,\xi_2)
=-\frac{1}{2}\left(1+\frac{2\pi(\xi_1\alpha_1+\xi_2\alpha_2)+m\beta}{(4\pi^2(\xi_1^2+\xi_2^2)+m^2)^{1/2}}\right)$$
and $(2\pi(\xi_1\alpha_1+\xi_2\alpha_2)+m\beta)^2
=4\pi^2(\xi_1^2+\xi_2^2)+m^2$, so the only eigenvalues of $\FF(\Lambda)(\xi_1,\xi_2)$ are $0$ and $-1$, and the corresponding spaces of eigenvectors with a fixed eigenvalue have the same dimension.
As a conclusion, $\{\Phi(h):\,h\in\Ker(\Lambda)\}\subsetneq D'$, and the second statement of $(i)(b)$ follows by Theorem \ref{t1}$(ii)$. The last statement of $(i)(b)$ follows essentially as Theorem \ref{t2}$(ii)$, we leave the details for the reader. This finishes the proof of $(i)(b)$.

In what respects to $(ii)$, assume that $\Sigma=\{x\in\R^3:\, |x|=1\}$ and we define 
\begin{equation*}
f_\lambda(r)=\left\{
\begin{split}
&(\lambda (1+m)-2m)\frac{e^{mr}-e^{-mr}}{mr}\quad\text{for }r<1,\\
&\big(\lambda(e^{2m}(m-1)+1+m)-2m(e^{2m}-1)\big)\frac{e^{-mr}}{mr}\quad\text{for }r>1,
\end{split}
\right.
\end{equation*}
which is real analytic for $r\neq1$ (even around $r=0$). Given $x\in\R^3$ we set $|x|=r$ and,
for $r\neq1$, we take
$$\varphi_\lambda(x)=\frac{-i}{m|x|}\left(im|x|f_\lambda(r),0,
x_3f'_\lambda(r),(x_1+ix_2)f'_\lambda(r)
\right)^t,$$
which belongs to $L^2(\mu)^4$.
A computation shows that, if $\lambda$ satisfies 
\begin{equation}\label{p2eq1}
\begin{split}
m^2\lambda^2+2\big((2m^2+2m+1)e^{-2m}-1\big)\lambda
-4m^2=0
\end{split}
\end{equation}
then $(H+V_\lambda)(\varphi_\lambda)=0$ distributionally, where $V_\lambda(\varphi_\lambda)
=\frac{\lambda}{2}((\varphi_\lambda)_++(\varphi_\lambda)_-)\sigma$ and $(\varphi_\lambda)_\pm$ denote the boundary values of $\varphi_\lambda$ when we approach non-tangentially to $\Sigma$ from inside/outside the ball $\Omega_+=\{x\in\R^3:\,|x|<1\}$ (we have chosen $N(x)=x/|x|$).
It is not hard to show that there exists some real  $\lambda\neq\pm 2$ satisfying (\ref{p2eq1}). For this $\lambda$, if we prove that the corresponding $\varphi_\lambda$ belongs to the domain $D(T)$ of Theorem
\ref{t2}$(i)$, then $2\neq|\lambda|=\lambda_j$ for some $j\in J$, thus there must exist some $\lambda_j\neq2$. Furthermore, by the definition of $\{\lambda_j\}_{j\in J}$ (see the proof of Theorem \ref{t2}$(i)$), if there is some $\lambda_j\neq2$ then $0\neq1/\lambda_j^2-1/4=a_j$ is an eigenvalue of $K$, thus $K$ is not identically zero, but since 
$K=C_\sigma(\alpha\cdot N)\{\alpha\cdot N,C_\sigma\}$ and $C_\sigma(\alpha\cdot N)$ is invertible by Lemma \ref{l5}$(ii)$, then $\{\alpha\cdot N,C_\sigma\}$ must not be identically zero. 

It only remains to check that $\varphi_\lambda\in D(T)$, where $D(T)$ is given by Theorem \ref{t2}. Using that $\varphi_\lambda$ decays exponentially at infinity and that $H(\varphi_\lambda)=0$ in $\Sigma^c$, one can verify that (\ref{l5eq5}) can be applied to $\varphi_\lambda$ either in $\Omega_+$ or $\Omega_-=\{x\in\R^3:\,|x|>1\}$. Therefore, using (\ref{l5eq5}) and Lemma \ref{l5}$(i)$, we have
$(\varphi_\lambda)_\pm=(1/2\pm iC_\sigma(\alpha\cdot N))(\varphi_\lambda)_\pm$, which implies that 
\begin{equation}\label{p2eq6}
\begin{split}
(\varphi_\lambda)_\pm=\pm2iC_\sigma(\alpha\cdot N)(\varphi_\lambda)_\pm.
\end{split}
\end{equation}
Set 
$g_\lambda=i(\alpha\cdot N)((\varphi_\lambda)_+-(\varphi_\lambda)_-)\in L^2(\sigma)^4$ and $\psi_\lambda=\Phi(g_\lambda)$. Then, from Lemma \ref{l5}$(i)$ and (\ref{p2eq6}), we deduce
\begin{equation*}
\begin{split}
(\psi_\lambda)_\pm
&=\Big(\mp\frac{i}{2}\,(\alpha\cdot N)+C_\sigma\Big)\big(i(\alpha\cdot N)((\varphi_\lambda)_+-(\varphi_\lambda)_-)\big)\\
&=\pm\frac{1}{2}\,(\varphi_\lambda)_+\mp\frac{1}{2}\,(\varphi_\lambda)_-
+iC_\sigma(\alpha\cdot N)(\varphi_\lambda)_+
-iC_\sigma(\alpha\cdot N)(\varphi_\lambda)_-
=(\varphi_\lambda)_\pm.
\end{split}
\end{equation*}
Hence, $\varphi_\lambda$ and $\psi_\lambda$ are two functions with the same boundary values on $\Sigma$ when we approach from $\Omega_\pm$ and they satisfy $H(\varphi_\lambda)=H(\psi_\lambda)=0$ so, by a uniqueness theorem in $\Omega_\pm$, we have $\varphi_\lambda=\psi_\lambda=\Phi(g_\lambda)$ in $L^2(\mu)^4$. Moreover, since $(H+V_\lambda)(\varphi_\lambda)=0$ distributionally, then
\begin{equation*}
\begin{split}
g_\lambda
&=-\frac{\lambda}{2}((\varphi_\lambda)_++(\varphi_\lambda)_-)
=-\lambda C_\sigma(g_\lambda),
\end{split}
\end{equation*}
which means that $(1+\lambda C_\sigma)(g_\lambda)=0$. Therefore, $\varphi_\lambda=\Phi(g_\lambda)\in D(T)$, and the proposition is finally proved.
\end{proof}

\begin{remark}
By using translations, rotations, and dilations, one can show that similar results hold for general planes and spheres in $\R^3$.
\end{remark}

\begin{teo}\label{t4}
Assume that $\Sigma$ is Lipschitz. Let $c\in\C$ and $\omega: L^2(\sigma)^4\to L^2(\sigma)^4$ be a bounded operator such that 
\begin{itemize}
\item[$(i)$] the commutator $[\omega,C_\sigma(\alpha\cdot N)]=\omega C_\sigma(\alpha\cdot N)-C_\sigma(\alpha\cdot N)\omega$ vanishes,
\item[$(ii)$] $\tau=I_4+i(1-2c)\omega+c(1-c)\omega^2$
is invertible in $L^2(\sigma)^4$,
\item[$(iii)$] $\Lambda=-(\alpha\cdot N)\tau^{-1}
\big(\omega+i(1/2-c)\omega^2-C_\sigma(\alpha\cdot N)\omega^2\big)$ is self-adjoint.
\end{itemize}
Set
$$D(T)=\big\{u+\Phi(g): u\in W^{1,2}(\mu)^4,\,g\in L^2(\sigma)^4,\,\Lambda(\Tr_\sigma(u))=g\big\}$$
and $T=H+V_\omega$ on $D(T)$, where 
$$V_\omega(\varphi)=(\alpha\cdot N)\omega(c\varphi_++(1-c)\varphi_-)\sigma$$ and 
$\varphi_\pm=\Tr_\sigma(u)+C_\pm (g)$.
Then $T:D(T)\subset L^2(\mu)^4\to L^2(\mu)^4$ is self-adjoint.
\end{teo}

\begin{proof}
Recall that 
$\varphi_\pm=\Tr_\sigma(u)\mp\frac{i}{2}\,(\alpha\cdot N)g
+C_\sigma(g)$ by Lemma \ref{l5}$(i)$. As before, if we want $V_\omega$ to coincide with the potential $V$ introduced in Corollary \ref{c3} (in order to apply Theorem \ref{t1}$(i)$), then we must have 
\begin{equation}\label{t4eq3}
-g=(\alpha\cdot N)\omega\big(\Tr_\sigma(u)+i(1/2-c)(\alpha\cdot N)g+C_\sigma(g)\big),
\end{equation}
which yields
\begin{equation}\label{t4eq1}
\begin{split}
-(\omega\Tr_\sigma)(u)&=(\alpha\cdot N)g+\omega\big(i(1/2-c)(\alpha\cdot N)+C_\sigma\big)(g)\\
&=\big(I_4+i(1/2-c)\omega+\omega C_\sigma(\alpha\cdot N)\big)((\alpha\cdot N)g).
\end{split}
\end{equation}
To shorten notation, we denote $\overline\omega=I_4+i(1/2-c)\omega$. Since $\omega C_\sigma(\alpha\cdot N)=C_\sigma(\alpha\cdot N)\omega$ by $(i)$ and $(C_\sigma(\alpha\cdot N))^2=-1/4$ by Lemma \ref{l5}$(ii)$, we easily deduce that
\begin{equation}\label{t4eq2}
\begin{split}
\big(\overline\omega-\omega C_\sigma(\alpha\cdot N)\big)
\big(\overline\omega+\omega C_\sigma(\alpha\cdot N)\big)
=\overline\omega^2+\omega^2/4.
\end{split}
\end{equation}
Notice that $\overline\omega^2+\omega^2/4=\tau$, which is invertible by $(ii)$. If we apply 
$\overline\omega-\omega C_\sigma(\alpha\cdot N)$ on both sides of (\ref{t4eq1}) and we use (\ref{t4eq2}), we obtain
\begin{equation}\label{t4eq4}
\begin{split}
-(\alpha\cdot N)\tau^{-1}
\big(\overline\omega-\omega C_\sigma(\alpha\cdot N)\big)(\omega\Tr_\sigma)(u)
=g.
\end{split}
\end{equation}
Observe that $-(\alpha\cdot N)\tau^{-1}
\big(\overline\omega-\omega C_\sigma(\alpha\cdot N)\big)\omega=\Lambda$, which is self-adjoint by $(iii)$, and $(\ref{t4eq4})$ can be rewritten as $\Lambda(\Tr_\sigma(u))=g$. Therefore, if $u+\Phi(g)\in D(T)$, then $u$ and $g$ satisfy (\ref{t4eq3}) by the construction of $\Lambda$, and Theorem \ref{t1}$(i)$ and Remark \ref{r2} show that $T:D(T)\to L^2(\mu)^4$ is self-adjoint. 
\end{proof}

\begin{teo}\label{t3}
Assume that $\Sigma$ is Lipschitz. Given $c\in\C$, there exists $\epsilon>0$ depending only on $\sigma$, $m$, and $c$ such that, if $\omega: L^2(\sigma)^4\to L^2(\sigma)^4$ is a bounded operator with $\|\omega\|_{L^2(\sigma)^4\to L^2(\sigma)^4}<\epsilon$,
$$\tau=I_4+\omega(i(1/2-c)(\alpha\cdot N)+C_\sigma)$$ is invertible in $L^2(\sigma)^4$.
Moreover, if $\tau^{-1}\omega$ is self-adjoint and we set 
$$D(T)=\big\{u+\Phi(g): u\in W^{1,2}(\mu)^4,\,g\in L^2(\sigma)^4,\,(\tau^{-1}\omega\Tr_\sigma)(u)=-g\big\}$$
and $T=H+V_\omega$ on $D(T)$, where 
$$V_\omega(\varphi)=\omega(c\varphi_++(1-c)\varphi_-)\sigma.$$
Then $T:D(T)\subset L^2(\mu)^4\to L^2(\mu)^4$ is self-adjoint.
\end{teo}

\begin{proof}
That $\tau$ is invertible follows easily from a Neumann serie argument, since the operator norm of $\alpha\cdot N$ is $1$ (to see it, use that it is self-adjoint in $L^2(\sigma)^4$ and satisfies $(\alpha\cdot N)^2=I_4$) and the norm of $C_\sigma$ only depends on $\sigma$ and $m$. In particular, $\epsilon$ can be taken so that $$\epsilon\geq1/2+|c|+\|C_\sigma\|_{L^2(\sigma)^4\to L^2(\sigma)^4}.$$

Assume that $\tau^{-1}\omega$ is self-adjoint. If $\varphi=u+\Phi(g)\in D(T)$, then 
$\varphi_\pm=\Tr_\sigma(u)\mp\frac{i}{2}\,(\alpha\cdot N)g
+C_\sigma(g)$ by Lemma \ref{l5}$(i)$. Hence 
\begin{equation*}
\begin{split}
V_\omega(\varphi)&=\omega\big(c\varphi_++(1-c)\varphi_-\big)\sigma
=\omega\big(\Tr_\sigma(u)+i(1/2-c)(\alpha\cdot N)g+C_\sigma(g)\big)\sigma\\
&=\big( (\omega\Tr_\sigma)(u)+\tau(g)-g\big)\sigma
=\big( (\omega\Tr_\sigma)(u)-\tau(\tau^{-1}\omega\Tr_\sigma)(u)-g\big)\sigma
=-g\sigma,\\
\end{split}
\end{equation*}
thus $V_\omega$ coincides with the potential $V$ introduced in Corollary \ref{c3}. Therefore, Theorem \ref{t1}$(i)$ and Remark \ref{r2} apply with $\Lambda=-\tau^{-1}\omega$, proving that $T:D(T)\to L^2(\mu)^4$ is self-adjoint. 
\end{proof}

\begin{remark}\label{r1}
Similar results to Lemma \ref{l5} and Theorems \ref{t4} and \ref{t3} hold when $\Sigma=\{(x_1,x_2,x_3)\in\R^3:\,x_3=A(x_1,x_2)\}$ is the graph of a Lipschitz function $A:\R^{2}\to\R$ (see \cite{McIntosh} for the case of Lemma \ref{l5}). We omit the details.
\end{remark}

\subsection{Some examples}\label{ss6}
We give some particular examples of potentials that fit in the last two theorems. Concerning Theorem \ref{t4}, we consider the following ones:
\begin{itemize}
\item[$(i)$] Take $\lambda\in\R$ and $\omega=\lambda I_4$, that is 
$$V_\omega(\varphi)=\lambda(\alpha\cdot N)(c\varphi_++(1-c)\varphi_-).$$ In this case, for $c=1/2$, $\tau= -\lambda^2/4-1$ is invertible for all $\lambda\in\R$, and then
\begin{equation*}
\Lambda=4\lambda(\lambda^2+4)^{-1}\big(\lambda(\alpha\cdot N)C_{\sigma}-1\big)(\alpha\cdot N)
\end{equation*}
is self-adjoint. 
\item[$(ii)$]  Set $\omega=rI_4+sC_\sigma(\alpha\cdot N)$ with $r,s\in \R$. The commutator $[\omega,C_\sigma(\alpha\cdot N)]$ vanishes and $\tau$ can be written as $pI_4+qC_\sigma(\alpha\cdot N)$, where $p=(2c-1)ir + c(c-1)\left(r^2+s^2/4\right)-1$ and $q=(2c-1)is + 2rsc(c-1)$. Notice that 
\begin{equation*}
(pI_4-qC_\sigma(\alpha\cdot N))(pI_4+qC_\sigma(\alpha\cdot N))=p^2-q^2/4. 
\end{equation*}
Hence, $\tau$ is invertible if $p^2\neq q^2/4$. It is easy to see that, for $c=1/2$, $p^2\neq q^2/4$ holds for all $r,s\in \R$. Therefore,  in this case, $[\omega,C_\sigma(\alpha\cdot N)]=0$ and $\tau$ is invertible. It is straightforward to check that then $\Lambda$ is self-adjoint in $L^2(\sigma)^4$.
\end{itemize}
In what respects to Theorem \ref{t3}, we consider the following potentials:
\begin{itemize}
\item[$(iii)$]  If we take $\omega=\lambda I_4$ with $\lambda\in \R$ small enough, that is 
$$V_\omega(\varphi)=\lambda(c\varphi_++(1-c)\varphi_-),$$
then $\tau=I_4+\lambda(i(1/2-c)(\alpha\cdot N)+C_\sigma)$ with $\Re{(c)}=1/2$ is invertible and self-adjoint, thus $\lambda\tau^{-1}$ is also self-adjoint.
\item[$(iv)$] By similar arguments it can be seen that, for 
$\omega=\delta\big(i  (1/2-c)(\alpha\cdot N)+C_\sigma\big)$ with $\delta\in \R$ small enough and $\Re{(c)}=1/2$, $\tau$ is self-adjoint and invertible, thus $\tau^{-1}\omega$ is self-adjoint.
\item[$(v)$] It is easy to see that any linear combination of the previous operators, say $\lambda I_4+\delta\big(i (1/2-c)(\alpha\cdot N)+C_\sigma\big)$, satisfy the assumptions of the theorem for $\lambda$ and $\delta$ small enough and $\Re{(c)}=1/2$.
\end{itemize}

\begin{remark}
Note the different nature of Theorems \ref{t4} and \ref{t3}, since the first one is based on a commutativity property and the second one on a smallness assumption. For example, for the potential $V_\omega(\varphi)=\lambda(c\varphi_++(1-c)\varphi_-)\sigma$, Theorem \ref{t4} can not be used because in this case $\omega=\lambda (\alpha\cdot N)$, which does not satisfy the assumption $(i)$ of the theorem. Indeed, for $\Sigma=\R^2\times\{0\}$, $\omega$ anticommutes with $C_{\sigma}(\alpha\cdot N)$.
\end{remark}

\end{document}